\documentclass[12pt,reqno]{amsart}
\usepackage{amsmath, amsthm, amssymb}
\usepackage{mathtools}

\usepackage[margin=1in]{geometry}
\usepackage{parskip}
\usepackage[shortlabels]{enumitem}

\usepackage{xcolor}

\usepackage[bookmarks,colorlinks,breaklinks]{hyperref}
\hypersetup{linkcolor=red,citecolor=blue,
  filecolor=dullmagenta,urlcolor=blue}

\usepackage[numbers]{natbib}
\usepackage{cleveref}

\usepackage{graphicx}
\usepackage{tikz}
\usetikzlibrary{matrix, arrows}
\usetikzlibrary{cd}

\makeatletter
\def\thm@space@setup{%
  \thm@preskip=0.5em\thm@postskip=\thm@preskip%
}
\makeatother

\newtheoremstyle{named}{}{}{\\itshape}{}{\bfseries}{.}{.5em}{\thmnote{#3's }#1}
\theoremstyle{named}

\theoremstyle{plain}
\newtheorem{thm}{Theorem}[section]
\newtheorem{prop}[thm]{Proposition}
\newtheorem{lem}[thm]{Lemma}
\newtheorem{cor}[thm]{Corollary}
\theoremstyle{definition}

\newtheorem{hyp}[thm]{Hypothesis}
\theoremstyle{remark}
\newtheorem{rmk}[thm]{Remark}
\crefformat{footnote}{#2\footnotemark[#1]#3}

\newcommand{\GL}[1]{\mathrm{GL}_{#1}}
\newcommand{\gl}[2]{\mathrm{GL}_{#1}({#2})}

\renewcommand{\sl}[2]{\mathrm{SL}_{#1}({#2})}
\newcommand{\SO}[1]{\mathrm{SO}_{#1}}
\newcommand{\so}[2]{\mathrm{SO}_{#1}({#2})}

\newcommand{\Sp}[1]{\mathrm{Sp}_{#1}}
\renewcommand{\sp}[2]{\mathrm{Sp}_{#1}({#2})}
\newcommand{\GSp}[1]{\mathrm{GSp}_{#1}}
\newcommand{\gsp}[2]{\mathrm{GSp}_{#1}({#2})}

\newcommand{\GSpin}[1]{\mathrm{GSpin}_{#1}}
\newcommand{\gspin}[2]{\mathrm{GSpin}_{#1}({#2})}

\newcommand{\brho}{\bar\rho}

\newcommand{\bkp}{\bar\kappa}

\newcommand{\tr}{\operatorname{tr}}

\newcommand{\gal}[2]{\mathrm{Gal}{(#1 / #2)}}
\newcommand{\Gal}[1]{\Gamma_{#1}}

\newcommand{\rats}{\mathbb{Q}}
\newcommand{\reals}{\mathbb{R}}

\newcommand{\cmplx}{\mathbb{C}}
\newcommand{\ints}{\mathbb{Z}}
\newcommand{\Qp}{\rats_p}
\newcommand{\Zp}{\ints_p}

\newcommand{\bQp}{\overline\rats_p}
\newcommand{\bZp}{\overline\ints_p}
\newcommand{\bQl}{\overline\rats_l}

\newcommand{\bFp}{\overline{\mathbb{F}}_p}

\newcommand{\bp}{\begin{pmatrix}}
\newcommand{\ep}{\end{pmatrix}}
\newcommand{\mc}{\mathcal}
\newcommand{\mf}{\mathfrak}
\newcommand{\mb}{\mathbb}
\newcommand{\mr}{\mathrm}

\newcommand{\frob}[1]{\mathrm{Fr}_{#1}}

\newcommand{\into}{\hookrightarrow}

\newcommand{\op}[1]{\operatorname{#1}}
\newcommand{\ov}{\overline}
\newcommand{\un}{\underline}

\makeatother

\title{Potential automorphy of $\op{GSpin}_{2n+1}$-valued Galois representations} 

\author[S.~Patrikis]{Stefan Patrikis}
\address{Department of Mathematics, The Ohio State University, 100 Math Tower, 231 W 18th Avenue, Columbus, OH 43210}
\author[S.~Tang]{Shiang Tang}
\address{Department of Mathematics, University of Illinois at Urbana-Champaign, 1409 W Green Street, Urbana, IL 61801}


\begin{document}

\maketitle 

\begin{abstract}
We prove a potential automorphy theorem for suitable Galois representations $\Gamma_{F^+} \to \mr{GSpin}_{2n+1}(\ov{\mb{F}}_p)$ and $\Gamma_{F^+} \to \mr{GSpin}_{2n+1}(\bQp)$, where $\Gamma_{F^+}$ is the absolute Galois group of a totally real field $F^+$. We also prove results on solvable descent for $\mr{GSp}_{2n}(\mb{A}_{F^+})$ and use these to put representations $\Gamma_{F^+} \to \mr{GSpin}_{2n+1}(\bQp)$ into compatible systems of $\mr{GSpin}_{2n+1}(\bQl)$-valued representations. 
\end{abstract}

\section{Introduction}
Given a connected reductive group $G$ defined over a number field $F$, the Langlands program predicts a connection between suitably algebraic automorphic representations of $G(\mb A_F)$ and geometric $p$-adic Galois representations $\mr{Gal}(\overline{F}/F) \to {}^LG(\bQp)$ into the L-group of $G$. Striking work of Kret-Shin (\cite{ks}) constructs the automorphic-to-Galois direction when $G$ is the group $\mr{GSp}_{2n}$ over a totally real field $F^+$, and $\pi$ is a cuspidal automorphic representation of $\mr{GSp}_{2n}(\mb{A}_{F^+})$ that is essentially discrete series at all infinite places and is a twist of the Steinberg representation at some finite place. In this paper we will establish a partial converse, proving a potential automorphy theorem, and some applications, for suitable $\mr{GSpin}_{2n+1}$-valued Galois representations. Before discussing our main results, we will put the work of Kret-Shin in context.

Their construction builds on two monumental works. First, it depends on the construction of automorphic Galois representations when $G= \mr{GL}_{2n+1}/F^+$, and $\pi$ is cuspidal, regular algebraic, essentially self-dual, and square-integrable at some finite place: extending work of Kottwitz, Clozel (\cite{clozel}) constructed the relevant Galois representations, and the essential properties for the purposes of \cite{ks} were proven by Harris-Taylor (\cite{ht}) and Taylor-Yoshida (\cite{taylor-yoshida}).\footnote{There is an extensive and deep literature devoted to eliminating the square-integrability condition: we mention as a sampling work of Bella\"{i}che, Caraiani, Chenevier, Clozel, Harris, Kottwitz, Labesse, Shin, and Taylor, which in turn relies on other deep automorphic advances, particularly work of Waldspurger and Ng\^{o}.} Second, it requires Arthur's endoscopic classification of representations (\cite{art13}), which among other things describes the discrete automorphic spectrum of $\mr{Sp}_{2n}(\mb{A}_{F^+})$ in terms of self-dual discrete automorphic representations of $\mr{GL}_{2n+1}(\mb{A}_{F^+})$.
These two marvelous developments allow (e.g., \cite[Theorem 2.4]{ks}) the construction of $\SO{2n+1}$-valued Galois representations associated to cuspidal automorphic representations of $G=\Sp{2n}$ that are discrete series at infinity and Steinberg at some finite prime (and indeed more generally).

Taking into account all of those advances, there remains a significant gap between the theorem of \cite{ks} for $\mr{GSp}_{2n}$ and the previously-known result for $\mr{Sp}_{2n}$. We briefly summarize the difficulty. Let $\Gamma_{F^+}= \mr{Gal}(\overline{F^+}/F^+)$, and fix an isomorphism $\iota \colon \cmplx \xrightarrow{\sim} \bQp$. A theorem of Tate shows that any representation $\Gamma_{F^+} \to \mr{SO}_{2n+1}(\bQp)$ lifts to a $\mr{GSpin}_{2n+1}(\bQp)$-valued representation, which is then determined up to a central twist. Thus one hopes to start with the automorphic representation $\tilde{\pi}$ of $\mr{GSp}_{2n}(\mb{A}_{F^+})$ under consideration, consider an automorphic component $\pi$ of its restriction to $\mr{Sp}_{2n}(\mb{A}_{F^+})$, construct the Galois representation $r_{\pi, \iota} \colon \Gamma_{F^+} \to \mr{SO}_{2n+1}(\bQp)$ associated to $\pi$, and then choose the ``right" lift to $\mr{GSpin}_{2n+1}(\bQp)$. As it stands, there is no technique internal to the theory so far described that allows one to do this: even if one can choose a lift with Clifford norm matching the central character of $\tilde{\pi}$, there remains at each place $v$ such that $\tilde{\pi}_v$ is unramified a $\pm 1$ ambiguity in whether the constructed Galois representation matches the Satake parameter of $\tilde{\pi}_v$ (not to speak of the other finite places). No elementary twisting argument can resolve this, and Kret-Shin address the problem by realizing the composition $\op{spin}(r_{\tilde{\pi}, \iota})$ of the desired $r_{\tilde{\pi}, \iota}$ with the (faithful) spin representation $\op{spin} \colon \mr{GSpin}_{2n+1} \to \mr{GL}_{2^n}$ inside the cohomology of a Shimura variety for a suitable inner form of $\mr{GSp}_{2n}$. The principal challenge of their paper consists of the subtle analysis of the cohomology of this Shimura variety.

We now return to the setup of our paper. The deepest inputs for our potential automorphy theorems are the potential automorphy theorem of Barnet-Lamb, Gee, Geraghty, Taylor of \cite{blggt} (and, for the strongest statement, a recent improvement due to Calegari, Emerton, and Gee in \cite{calegari-emerton-gee}) and Arthur's work (\cite{art13}). We develop these and their relationships with the construction of \cite{ks} to prove the following two potential automorphy results, one for mod $p$ and one for $p$-adic representations. Here and throughout the paper, we let $\op{std} \colon \mr{GSpin}_{2n+1} \to \mr{GL}_{2n+1}$ denote the standard representation, and we let $N \colon \mr{GSpin}_{2n+1} \to \mb{G}_m$ denote the Clifford norm. 
\begin{thm}[See Theorem \ref{residual aut}]\label{potautintro}
Let $p$ be a prime, $p \geq 2n+4$, and let $\bar{r} \colon \Gamma_{F^+} \to \mr{GSpin}_{2n+1}(\overline{\mb{F}}_p)$ be a continuous representation satisfying the following hypotheses:
\begin{itemize}
    \item The restriction $\op{std}(\bar{r})|_{\Gamma_{F^+(\mu_p)}}$ is irreducible.
    \item $\bar{r}$ is odd (see Hypothesis \ref{hypotheses}).
\end{itemize}
Then there exist a totally real Galois extension $L^+/F^+$ and a cuspidal automorphic representation $\tilde{\pi}$ of $\mr{GSp}_{2n}(\mb{A}_{L^+})$, satisfying the hypotheses (St) and (L-coh) of \cite{ks}, such that a suitable $\mr{GSpin}_{2n+1}$-conjugate of the representation $r_{\tilde{\pi}, \iota} \colon \Gamma_{L^+} \to \mr{GSpin}_{2n+1}(\bQp)$ constructed by \cite{ks} reduces mod $p$ to $\bar{r}|_{\Gamma_{L^+}}$.
\end{thm}
We note that this theorem makes no ``Steinberg" local hypothesis; the first step in its proof is to show (Theorem \ref{lifting thm}) that $\bar{r}$ admits a geometric lift $r \colon \Gamma_{F^+} \to \mr{GSpin}_{2n+1}(\bZp)$ that at some auxiliary finite place looks like the Langlands parameter of the Steinberg representation (up to twist). Theorem \ref{lifting thm} relies on using the lifting method of Khare-Wintenberger (\cite{kw}) and a comparison of deformation rings for the groups $\mr{GSpin}_{2n+1}$, $\mr{SO}_{2n+1}$, and $\mr{GL}_{2n+1}$, ultimately invoking a finiteness result from \cite{blggt} that itself comes from combining $R=T$ theorems with the method of \cite{kw}. Having constructed this lift $r$, we can run the method of the proof of our $p$-adic potential automorphy theorem, which we now state. For details of the terminology, see \S \ref{sec:notation} and \cite{calegari-emerton-gee}. 
\begin{thm}[Theorem \ref{padicpotaut}]\label{padicpotautintro}
Let $r \colon \Gamma_{F^+} \to \mr{GSpin}_{2n+1}(\bZp)$ be a continuous representation whose Clifford norm is a geometric character $N(r) \colon \Gamma_{F^+} \to \bZp^\times$. Assume the following:
\begin{itemize}
    \item $p>2(2n+1)$.
    \item $\bar{r}$ is odd.
    \item For some finite place $v_{\mr{St}}$ of $F^+$, the Frobenius semisimple Weil-Deligne representation associated to $r|_{\Gamma_{F^+_{v_{\mr{St}}}}}$ is equivalent to a twist of the Steinberg-type Langlands parameter (equivalently, $P(r|_{\Gamma_{F^+_{v_{\mr{St}}}}})$ is a Steinberg parameter for $\mr{Sp}_{2n}(F^+_{v_{\mr{St}}})$). 
    \item There exist a quadratic CM extension $F/F^+$ and a character $\mu \colon \Gamma_{F^+} \to \bZp^\times$ such that 
    \begin{itemize}
    \item $F$ does not contain $\zeta_p$, and $\op{std}(\bar{r})|_{\Gamma_{F(\zeta_p)}}$ is irreducible.  
    \item $(\op{std}(r)|_{\Gamma_{F}}, \mu)$ is polarized, and for some (any) choice of prolongation 
    \[
    \rho(\op{std}(\bar{r})|_{\Gamma_{F}}, \mu) \colon \Gamma_{F^+} \to \mc{G}_{2n+1}(\bZp),
    \]
$\rho(\op{std}(r), \mu)|_{\Gamma_{F^+_v}}$ is globally realizable for each $v \vert p$. Here $\mc G_{2n+1}$ denotes the Clozel-Harris-Taylor group scheme, see Section \ref{sec:notation}.
    \end{itemize}
\end{itemize}
Then there exist a totally real Galois extension $L^+/F^+$ and a cuspidal automorphic representation $\tilde{\pi}$ of $\mr{GSp}_{2n}(\mb{A}_{L^+})$ satisfying the hypotheses (St) and (L-coh) of \cite{ks} such that $r_{\tilde{\pi}, \iota}$ is equivalent to $r|_{\Gamma_{L^+}}$ as $\mr{GSpin}_{2n+1}(\bQp)$-representations.
\end{thm}
For instance, the hypothesis on places above $p$, which we should note forces the Hodge-Tate cocharacters of $r|_{\Gamma_{F^+_v}}$ to be regular, includes the cases where $\op{std}(r)|_{\Gamma_{F^+_v}}$ is Hodge-Tate regular and potentially diagonalizable in the sense of \cite{blggt}. In that case the theorem statement admits a straightforward modification with no need for an auxiliary extension $F$ or mention of prolongations. 

The strategy here is to combine the potential automorphy theorem of \cite{calegari-emerton-gee} 
with \cite{art13} to realize the projection $P(r) \colon \Gamma_{F^+} \to \mr{SO}_{2n+1}(\bZp)$, after restriction to some $L^+$, as the representation $r_{\pi, \iota}$ associated to a cuspidal automorphic representation of $\pi$ of $\mr{Sp}_{2n}(\mb{A}_{L^+})$. Then we show that $\pi$ can be extended to a cuspidal automorphic representation $\tilde{\pi}$ of $\mr{GSp}_{2n}(\mb{A}_{L^+})$ that satisfies the hypotheses of the main theorem of \cite{ks}, and finally we check that a suitable twist of $\tilde{\pi}$ in fact corresponds to our original $r|_{\Gamma_{L^+}}$.   

In \S \ref{sec:compatible}, we discuss an application of our potential automorphy theorem. Recall that for $\mr{GL}_n$, a method due to Taylor (\cite{taylor:galoisreps}) based on combining potential automorphy theorems with Brauer's theorem from finite group theory allows one to realize a representation $r \colon \Gamma_{F^+} \to \mr{GL}_n(\bQp)$, assumed to satisfy the hypotheses of the $\mr{GL}_n$ potential automorphy theorem, as part of a (strictly) compatible system of $\ell$-adic Galois representations. We apply Theorem \ref{padicpotaut} to prove an analogous result for $\mr{GSpin}_{2n+1}(\bQp)$-representations. This may be the subtlest part of our work. Taylor's method relies on solvable descent, which depends on a combination of cyclic (prime degree) descent, established in \cite{arthur-clozel}, and knowledge that in the relevant cases the descended automorphic representations themselves have associated Galois representations. We establish instances of solvable descent for $\mr{GSp}_{2n}(\mb{A}_{L^+})$, 
bootstrapping as before from the case of $\mr{GL}_{2n+1}$. For our principal result, Theorem \ref{compsystem}, we only need the solvable descent of an invariant automorphic representation in a setting where the associated Galois representation is already known to descend; we treat this problem in Theorem \ref{solvabledescent}. We conclude the section with a complementary result (Corollary \ref{descentinv}) when only the automorphic representation is known to be invariant:  then cyclic descent still holds in our setting, but solvable descent is complicated, as for $\mr{GL}_N$, by the possibility of invariant Hecke characters that do not descend.

The other difficulty in extending Taylor's argument is that of course no Brauer argument can be made directly with $\mr{GSpin}_{2n+1}$-valued Galois representations. The particular possibilities for the algebraic monodromy groups of the representations $r_{\tilde{\pi}, \iota}$ constructed by \cite{ks} make possible a case-by-case analysis, with the ``generic" case of full $\mr{GSpin}_{2n+1}$ monodromy relying on applying Brauer's theorem to $\op{spin}(r)$. (That said, we emphasize that the potential automorphy theorem being applied is still for $r \colon \Gamma_{F^+} \to \mr{GSpin}_{2n+1}(\bQp)$, not for $\op{spin}(r)$; indeed, the latter will typically not satisfy the hypotheses of $\mr{GL}_{2^n}$ potential automorphy theorems.) Theorem \ref{compsystem} combines these ideas to construct the desired compatible system:
\begin{thm}\label{compsystemintro}
Let $r \colon \Gamma_{F^+} \to \mr{GSpin}_{2n+1}(\bZp)$ satisfy the hypotheses of Theorem \ref{padicpotautintro}. Then for all primes $\ell$ and choices of isomorphism $\iota_\ell \colon \cmplx \xrightarrow{\sim} \bQl$ there is a continuous representation $r_{\iota_{\ell}} \colon \Gamma_{F^+} \to \mr{GSpin}_{2n+1}(\bQl)$ such that:
\begin{itemize}
    \item For all but finitely many primes $v$ at which $r$ is unramified, the semisimple conjugacy class of $\iota^{-1}r(\frob{v})^{\mr{ss}}$ agrees with that of $\iota_{\ell}^{-1}r_{\iota_{\ell}}(\frob{v})^{\mr{ss}}$.
    \item For all primes $v \vert \ell$, $r_{\iota_{\ell}}$ is de Rham, and its Hodge-Tate cocharacters are determined up to conjugacy by those of $r$: for each embedding $\tau \colon F^+ \to \bQl$, determining a place $v \vert \ell$ of $F^+$, and inducing $\iota \iota_{\ell}^{-1} \tau \colon F^+ \to \bQp$ and a corresponding place $\iota_{\ell}^*(v) \vert p$ of $F^+$, we have the equality  $\iota_{\ell}^{-1}\mu(r_{\iota_{\ell}}|_{\Gamma_{F^+_v}}, \tau)= \iota^{-1}\mu(r|_{\Gamma_{F^+_{\iota_{\ell}^*(v)}}}, \iota \iota_{\ell}^{-1} \tau)$ of conjugacy classes of $\mr{GSpin}_{2n+1}(\cmplx)$-valued cocharacters.
\end{itemize}
\end{thm}

Finally, we remark on what we cannot prove at present. While our mod $p$ potential automorphy result imposes no Steinberg local condition, the $p$-adic potential automorphy theorem does, and we do not know how to circumvent this. Certainly a corresponding improvement in the results of \cite{ks}, constructing $r_{\tilde{\pi}, \iota}$ for all L-cohomological cuspidal automorphic representations of $\mr{GSp}_{2n}(\mb{A}_{F^+})$, would allow us to extend these results to that generality. Even with such an extension, generalizing the results of \S \ref{sec:compatible} would be a more serious task.
\emph{Acknowledgments:} S.P. was supported by NSF grants DMS-1700759 and DMS-1752313.
S.P. thanks Wushi Goldring for enjoyable discussions related to the subject of this paper, and S.T. thanks Patrick Allen for many helpful conversations and for answering many of his questions. We are very grateful to an anonymous referee for a careful and critical reading of the first version of this paper.

\section{Notation and Conventions}\label{sec:notation}
Let $F$ be a field. Fix an algebraic closure $\ov{F}$ of $F$ and write $\Gamma_F$ for the absolute Galois group $\gal{\ov{F}}{F}$ of $F$. If $F$ is a number field, then for each place $v$ of $F$, we fix an embedding $\ov{F} \to \ov{F_v}$ into an algebraic closure of $F_v$, which gives rise to an injective group homomorphism $\Gamma_{F_v} \to \Gamma_F$. 
For any finite place $v$, let $k_v$ be the residue field of $v$ and let $\frob{v} \in \Gamma_{k_v}$ be the arithmetic Frobenius.
If $H$ is a group (typically the points over a finite field or a $p$-adic field of a reductive algebraic group), and there is a continuous group homomorphism $r: \Gamma_F \to H$, we will sometimes write $r|_v$ 
for $r|_{\Gamma_{F_v}}$, the restriction of $r$ to the decomposition group $\Gamma_{F_v}$. If $H$ acts on a finite-dimensional vector space $V$, we write $r(V)$ for the $\Gamma_F$-module induced by precomposing this action with $r$. (Typically $H$ will be a reductive algebraic group and $V$ will be its Lie algebra equipped with the adjoint action of $H$.) Let $\kappa: \Gamma_F \to \bZp^{\times}$ be the $p$-adic cyclotomic character and $\bkp$ be its reduction modulo $p$. We will always assume $p \neq 2$, and our main theorems will make stronger hypotheses on $p$.

Let $n \geq 1$ be an integer, let $\SO{2n+1}$ be the odd orthogonal group of rank $n$, which for convenience we will take to be defined by $\mr{SO}_{2n+1}(R)= \{g \in \mr{SL}_{2n+1}(R): {}^t g \cdot g=1\}$ for any ring $R$. In fact, we will always be free to enlarge the ring of definition and consider this group to be defined over the ring of integers $\mc{O}$ in a sufficiently large finite extension of $\Qp$, and in particular we will assume this is in fact the split form of the group. Let $\GSpin{2n+1}$ be the corresponding general spin group, so there is an exact sequence of algebraic groups
\[
1 \to Z \to \GSpin{2n+1} \xrightarrow{P} \SO{2n+1} \to 1
\]
where $Z \cong \mb G_m$ is the center of $\GSpin{2n+1}$. 
Let $N: \GSpin{2n+1} \to \mb G_m$ be the Clifford norm.
Note that $\gspin{2n+1}{\cmplx}$ is the Langlands dual of $\GSp{2n}$.
Let $\op{std}: \GSpin{2n+1} \to \GL{2n+1}$ be the composition of $P$ with the standard representation of $\mr{SO}_{2n+1}$.
Let $\op{spin}: \GSpin{2n+1} \to \GL{2^n}$ be the spin representation. For a homomorphism $r: \Gamma \to \GSpin{2n+1}$, we will frequently consider the composites $P(r): \Gamma \to \SO{2n+1}$ and $\op{std}(r): \Gamma \to \op{GL}_{2n+1}$. 

We recall the definition of the Clozel-Harris-Taylor group scheme $\mc G_n$ over $\ints$ which is defined as the semidirect product $(\GL{n} \times \GL{1}) \rtimes \{1,\jmath\}$ where $\jmath(g,a)\jmath=(a({}^t g)^{-1},a)$, and the similitude character $\nu: \mc G_n \to \GL{1}$ given by $\nu(g,a)=a$ and $\nu(\jmath)=-1$. The groups $\mc{G}_n$ will arise in the following setting. Suppose we have a homomorphism $r \colon \Gamma_F \to \mr{GSpin}_{2n+1}(R)$ for some field $F$ and some ring $R$. 
Let $F'/F$ be a quadratic extension of $F$, and define $\rho(r) \colon \Gamma_F \to \mc{G}_{2n+1}(R)$ as the composite
\[
\Gamma_F \xrightarrow{P(r) \times \mr{res}_{F'}} \mr{SO}_{2n+1}(R) \times \mr{Gal}(F'/F) \to \mc{G}_{2n+1}(R),
\]
where the last map sends $g\in \so{2n+1}{R}$ to its image in $\gl{2n+1}{R}$ and sends the nontrivial element of $\gal{F'}{F}$ to $\jmath$. 
By our choice of definition of $\mr{SO}_{2n+1}(R)$, $\rho(r)$ is a well-defined homomorphism. 
We will refer to $\rho(r)$ as the standard prolongation of $\op{std}(r)$ with respect to the extension $F'/F$ (compare \cite[\S 1.4]{calegari-emerton-gee}).

Let $T$ be a maximal torus of a reductive group $G$ over an algebraically closed field. A cocharacter $\mu \in X_*(T)$ is said to be regular if $<\mu,\alpha> \neq 0$ for all roots $\alpha$ of $T$ in $G$. 

We assume that the reader is familiar with Galois deformation theory and will use the standard terms and results of Galois deformation theory freely throughout this paper. The reader may refer to \cite[Section 1]{blggt} and \cite{bg:Gdef} for the details. We recall here some deformation-theoretic terminology. 
Given a topologically finite-generated profinite group $\Gamma$, a finite extension $E/\Qp$ with ring of integers $\mc O$ and residue field $k$, a reductive algebraic group $G$ defined over $\mc O$ and a continuous homomorphism $\bar{r}: \Gamma \to G(k)$, let $R^{\Box}_{\mc{O}, \bar{r}}$ be the universal lifting ring representing the functor sending a complete local noetherian $\mc O$-algebra $R$ with residue field $k$ to the set of lifts $r: \Gamma \to G(R)$ of $\bar{r}$. We will always leave the $\mc{O}$ implicit, writing only $R^{\Box}_{\bar{r}}$, and at various points in the argument we enlarge $\mc{O}$; see \cite[Lemma 1.2.1]{blggt} for a justification of (the harmlessness of) this practice. We write $R_{\brho}^{\Box} \otimes \bQp$ for $R_{\mc O, \brho}^{\Box} \otimes_{\mc O} \bQp$ for any particular choice of $\mc{O}$, and again by \cite[Lemma 1.2.1]{blggt}, $R_{\brho}^{\Box} \otimes \bQp$ is independent of the choice of $E$.

Suppose moreover that $\chi \colon G \to A$ is a morphism to an abelian group scheme over $\mc{O}$, and $\theta: \Gamma \to \mc O^{\times}$ is a continuous character whose mod $p$ reduction equals $\chi(\bar{r}): \Gamma \to k^{\times}$. Then we write $R^{\Box, \theta}_{\bar{r}}$ (again, leaving $\mc O$ implicit) for the quotient of $R^{\Box}_{\bar{r}}$ representing the functor sending a complete local noetherian $\mc O$-algebra $R$ with residue field $k$ to the set of liftings $r: \Gamma \to G(R)$ of $\bar{r}$ with fixed similitude character $\theta$, i.e. $\chi(r)=\theta$. We will apply this convention when the character $\chi$ is the Clifford norm $N \colon \mr{GSpin}_{2n+1} \to \mb{G}_m$, and when it is the character $\nu \colon \mc{G}_{2n+1} \to \mb{G}_m$ defined above.

When $K/\Qp$ is a finite extension and $\Gamma$ is $\Gamma_K$, we moreover consider quotients of $R^{\Box, \theta}_{\bar{r}}$ having fixed inertial type and $p$-adic Hodge type. The fundamental analysis here is due to Kisin (\cite{kisin:pst}), and the state of the art, and our point of reference, is \cite{bg:Gdef}, and we refer there for details. We will index $p$-adic Hodge types of deformations $r \colon \Gamma_K \to G(\mc{O})$ by collections $\underline{\mu}(r)=\{\mu(r, \tau)\}_{\tau \colon K \to \bQp}$ of (conjugacy classes of) Hodge-Tate co-characters, and write $R^{\Box, \theta, \underline{\mu}(r)}_{\bar{r}}$ for the $\Zp$-flat quotient of $R^{\Box, \theta}_{\bar{r}}$ whose points in finite local $E$-algebras are precisely those of $R^{\Box, \theta}_{\bar{r}}$ that are moreover potentially semi-stable with $p$-adic Hodge type $\underline{\mu}(r)$. We likewise consider the quotients with fixed inertial type $\sigma$, $R^{\Box, \theta, \un{\mu}(r), \sigma}_{\bar{r}}$, referring to \cite[\S 3.2]{bg:Gdef} for details.  

We fix, \emph{once and for all}, a field isomorphism $\iota: \cmplx \xrightarrow{\sim} \bQp$. This will be used in associating $p$-adic Galois representations to automorphic representations. We will typically note the use of this isomorphism explicitly, with one exception: for type $A_0$ Hecke characters $\chi \colon \mathbb{A}_F^\times/F^\times \to \cmplx^\times$ of a number field $F$, we will continue to write $\chi \colon \Gamma_F \to \bQp^\times$ for the $p$-adic Galois character associated to $\chi$ (and similarly for characters of localizations $F_v^\times$). Thus, we leave the isomorphism $\iota$ implicit and the global (or local) reciprocity map implicit; we have done this to simplify the notation, since we do not expect it to cause any confusion. These local and global reciprocity maps are normalized as in \cite{ks}, which is to say geometric Frobenii correspond to uniformizers.

\section{A lifting theorem}\label{sec:lifting}
We prove a geometric lifting theorem for $\GSpin{2n+1}(\bFp)$-valued representations analogous to \cite[Theorem 4.3.1]{blggt} or \cite[Theorem 4.2.11]{calegari-emerton-gee}. We will use a special case of this theorem in \S \ref{sec:automorphy} as a first step in establishing our mod $p$ potential automorphy theorem. In this section we will prove something rather more general than what is needed for \S \ref{sec:automorphy}, although we have not attempted to optimize the results. The lifting theorem has global hypotheses---that the mod $p$ representation be odd and irreducible, essentially---and then a local hypothesis on the existence of suitable de Rham lifts at places above $p$. A similar theorem, with less restrictive local hypotheses, follows from \cite[Theorem A]{fkp:reldef}, but the methods of that paper cannot yield the minimal lifts produced here via potential automorphy (nor of course can it show they are potentially automorphic).

We will freely make use of the definition of ``globally realizable" components of potentially semistable local deformation rings from \cite[Definition 2.1.9]{calegari-emerton-gee}, and we refer the reader there for details.
Let $K^+$ be a totally real field, and let $\bar r: \Gamma_{K^+} \to \gspin{2n+1}{k}$ be a continuous representation, where $k$ is the residue field of the ring of integers $\mc{O}$ in a finite extension $E$ of $\Qp$ (which we will feel free to enlarge without comment). Fix a geometric lift $\theta: \Gamma_{K^+} \to \bZp^{\times}$ of $N \circ \bar r: \Gamma_{K^+} \to \bFp^{\times}$. 
\begin{hyp}\label{hypotheses}
Assume that $\bar{r}$ satisfies the following hypotheses. There exists a CM quadratic extension $K/K^+$, not containing $\zeta_p$, such that 
\begin{enumerate}
\item The restriction $\op{std}(\bar r)|_{\Gal{K(\mu_p)}}$ is irreducible.\footnote{We note that this is strictly stronger than assuming $\bar{r}|_{\Gal{K(\mu_p)}}$ is irreducible. Compare \cite{fkp:Girr}.}
\item The prime $p$ satisfies $p > 2(2n+1)$.
\item The representation $\bar r$ is odd, meaning that for all infinite places $v$ of $K^+$, with corresponding conjugacy class of complex conjugations $c_v \in \Gamma_{K^+}$, $\bar r(c_v)$ is a Chevalley involution in $\GSpin{2n+1}$, i.e. $\dim \mf g^{\mr{Ad}(\bar r(c_v))=1}=(\dim \mf g)/(\dim \mf b)$, where $\mf g \cong \mf{so}_{2n+1}$ is the commutator subalgebra of the Lie algebra of $\GSpin{2n+1}$, equipped with an action of $\Gamma_{K^+}$ induced by $\bar r$ and the adjoint representation, and $\mf b \subset \mf g$ is any Borel subalgebra of $\mf g$.
\item 
For each $v \vert p$, there is a lift $r_v \colon \Gal{K^+_v} \to \mr{GSpin}_{2n+1}(\bZp)$ of $\bar r|_{K^+_v}$ such that 
\begin{itemize}
\item $N \circ r_v=\theta|_{\Gamma_{K^+_v}}$.
\item $r_v$ is de Rham, with some $p$-adic Hodge type $\un{\mu}(r_v)$ and some inertial type $\sigma_v$, and there is an irreducible component $\mc{C}(r_v)$ of $R^{\Box, \theta, \un{\mu}(r_v), \sigma_v}_{\bar{r}}[1/p]$ containing $r_v$ that, under the map of lifting rings 
\[
\op{Spec} R_{\bar{r}}^{\Box, \theta, \un{\mu}(r_v), \sigma_v} \to \op{Spec} R^{\Box, \delta_{K/K^+}, \un{\mu}(\rho(r_v)), \rho(\sigma_v)}_{\rho(\bar{r})} 
\]
induced by the standard prolongation (with respect to $K/K^+$), maps to a globally realizable component $\mc{C}(\rho(r_v))$. Here $\delta_{K/K^+}$ is the unique nontrivial quadratic character of $\Gal{K^+}$ that is trivial on $\Gal{K}$.
\item In particular, the previous item implies that, for each embedding $\tau: K^+_v \to \bQp$, the Hodge-Tate cocharacters $\mu(\op{std}(r_v), \tau)$ are regular, and consequently the Hodge-Tate cocharacters $\mu(r_v, \tau)$ are regular.
\end{itemize}
\end{enumerate}
\end{hyp}
\begin{rmk}\label{blggtlift}
If we were basing our arguments on the less general results of \cite{blggt}, then here we would instead assume $\op{std}(r_v)$ is regular and potentially diagonalizable in the sense of \cite[\S 1.4]{blggt}. We would then have a version of the above hypotheses without mention of the extension $K$, and we would choose $K$ not containing $\zeta_p$ such that $\op{std}(\bar{r})|_{\Gal{K(\mu_p)}}$ remained irreducible, and such that all finite primes in $S$ split in $K/K^+$. We would also only need to assume $p \geq 2n+4$, in order to apply the theorems of \cite{blggt} and Lemma \ref{defcompare}. Then we would obtain a version of Theorem \ref{lifting thm} below by invoking \cite[Theorem 4.3.1]{blggt} instead of \cite{calegari-emerton-gee} in the arguments below. To see that this potentially diagonalizable case is contained in our current arguments, we would choose $\mc{C}(r_v)$ to be a potentially crystalline component (if $\op{std}(r_v)$ is potentially crystalline, then so is $r_v$, but $r_v$ could conceivably lie at an intersection of crystalline and semi-stable components), mapping into a potentially crystalline component containing the potentially diagonalizable point $\rho(r_v)$. Any potentially diagonalizable point lies on some globally realizable component (\cite[Remark 2.1.14]{calegari-emerton-gee}), which (by the argument given to justify this assertion) is moreover a potentially crystalline component. Since potentially crystalline deformation rings are regular, there are no other potentially crystalline components passing through $\rho(r_v)$. Hence, whatever potentially crystalline component $\mc{C}(r_v)$ we chose must map to a globally realizable component $\mc{C}(\rho(r_v))$, namely the unique potentially crystalline component passing through $\rho(r_v)$.
\end{rmk}
Let $S$ be a finite set of primes of $F^+$ containing the primes above $p$ and all primes where $\bar{r}$ or $\theta$ is ramified. To complement our assumption on the existence of $r_v$ for $v \vert p$, we have the following lemma:
\begin{lem}\label{booherlemma}
For all places $v \in S$ not above $p$, there is a lift $r_v \colon \Gamma_{K^+_v} \to \mr{GSpin}_{2n+1}(\bZp)$ of $\bar{r}|_{\Gamma_{K^+_v}}$ having $N(r_v)= \theta|_{\Gamma_{K^+_v}}$.
\end{lem}
\begin{proof}
Since $p \neq 2$, the natural map of lifting rings $R^{\Box}_{P(\bar{r})|_{\Gamma_{K^+_v}}} \to R^{\Box, \theta}_{\bar{r}|_{\Gamma_{K^+_v}}}$ is an isomorphism (see Lemma \ref{defcompare} for more details), so it suffices to find a lift of $P(\bar{r}|_{\Gamma_{K^+_v}})$. Such a lift exists by Booher's thesis (\cite{boo16}).
\end{proof}
From now on, we will fix some such lift $r_v$. 
We reiterate that we let $\mc{O}$ be the ring of integers in a finite extension $E$ of $\Qp$, enlarged if necessary so that all of the above data are defined over $\mc{O}$. 

We now state the main result of this section:
\begin{thm}\label{lifting thm}
Retain the assumptions of Hypothesis \ref{hypotheses}. Let $S$ be a finite set of places of $K^+$ such that $S$ contains all places above $p$ and all places where $\theta$ or $\bar{r}$ is ramified. Fix lifts $r_v$ and components $\mc{C}(r_v)$ for $v \vert p$ as in Hypothesis \ref{hypotheses}, and fix lifts $r_v$ for $v \in S \setminus \{v \vert p\}$ as in Lemma \ref{booherlemma}. Then $\bar r$ has a lift $r: \Gamma_{K^+} \to \gspin{2n+1}{\bQp}$ unramified outside $S$ such that
\begin{itemize}
    \item $N(r)= \theta$.
    \item For each $v\in S$ not lying above $p$, $r|_v:= r|_{\Gamma_{K^+_v}}$ and $r_v$ lie on the same irreducible component of $\op{Spec}(R^{\Box, \theta}_{\bar r|_v} \otimes \bQp)$.
    \item For each $v \vert p$, $r|_v$ lies on $\mc{C}(r_v)$,
    and in particular is potentially semistable with the given Hodge-Tate cocharacters and inertial types. 
\end{itemize} 
\end{thm}
The rest of this section will be dedicated to the proof of Theorem \ref{lifting thm}, using a version of the lifting argument of Khare-Wintenberger (\cite{kw}). We will define global deformation conditions for $\bar r$, $P(\bar r)$, and the standard prolongation $\rho(\bar{r})$ with respect to $K/K^+$. Recall that $\rho(\bar{r}): \Gamma_{K^+} \to \mc G_{2n+1}(\bFp)$ is the composite 
\[
\Gamma_{K^+} \xrightarrow{P(\bar r) \times res_K} \so{2n+1}{\bFp} \times \gal{K}{K^+} \to \mc G_{2n+1}(\bFp).
\]
Then we will show that there are natural finite maps between the corresponding deformation rings and use the $\Zp$-finiteness of the deformation ring for the group $\mc G_{2n+1}$ 
to conclude that the deformation ring for the group $\GSpin{2n+1}$ is $\Zp$-finite; this, combined with a standard calculation of its Krull dimension (as in \cite{bg:Gdef}), will imply that it has a $\bQp$-point.

Recall from the discussion in \S \ref{sec:notation} that for each $v\in S$ we can consider the lifting rings $R^{\Box, \theta}_{\bar r|_v}$, $R^{\Box}_{P(\bar r)|_v}$, and 
$R^{\Box, \delta_{K/K^+}}_{\rho(\bar{r})|_v}$.
For $v \in S$ not above $p$, choose an irreducible component $\mc{C}(r_v)$ (resp. $\mc{C}(P(r_v))$, 
$\mc{C}(\rho(r_v))$)
of $R^{\Box, \theta}_{\bar{r}|_v} \otimes \bQp$ (resp. $R^{\Box}_{P(\bar {r})|_v} \otimes \bQp$, 
$R^{\Box, \delta_{K/K^+}}_{\rho(\bar{r})|_v} \otimes \bQp$) containing $r_v$ (resp. $P(r_v)$, 
$\rho(r_v)$)
such that under the natural maps 
\[ 
\op{Spec} R^{\Box, \theta}_{\bar r|_v} \otimes \bQp \to \op{Spec} R^{\Box}_{P(\bar r)|_v}\otimes \bQp \to 
\op{Spec} R^{\Box, \delta_{K/K^+}}_{\rho(\bar{r})|_v} \otimes \bQp,
\]
$\mc C(r_v)$ maps to $\mc C(P(r_v))$ and $\mc C(P(r_v))$ maps to 
$\mc{C}(\rho(r_v))$. 

Similarly, for $v \vert p$, the fixed inertial type and fixed $p$-adic Hodge type of $r_v$ induces corresponding data for $P(r_v)$ and 
$\rho(r_v)$, and the fixed component $\mc{C}(r_v)$ maps to components $\mc{C}(P(r_v))$ of $R^{\Box, \un{\mu}(P(r_v)), P(\sigma_v)}_{P(\bar{r})|_v}\otimes \bQp$ and $\mc{C}(\rho(r_v))$ of
$R^{\Box, \delta_{K/K^+}, \un{\mu}(\rho(r_v)), \rho(\sigma_v)}_{\rho(\bar{r})|_v}\otimes \bQp$ (the latter as in the theorem hypotheses).

We now define three global deformation rings, for $\bar{r}$, $P(\bar{r})$, and $\rho(\bar{r})$, by considering lifts that locally lie on the irreducible components we have just specified. More precisely, following the formalism of \cite[\S 4.2]{bg:Gdef}, we let $R^{\mr{univ}}_{\mr{GSpin}}$ be the quotient of the universal, unramified outside $S$, fixed Clifford norm $\theta$, deformation ring for $\bar{r}$ corresponding to the fixed set of components $\{\mc{C}(r_v)\}_{v \in S}$. We similarly define $R^{\mr{univ}}_{\mr{SO}}$, corresponding to the local components $\{\mc{C}(P(r_v))\}$ and $R^{\mr{univ}}_{\mr{GL}}$ corresponding to the local components $\{\mc{C}(\rho(r_v))\}$ (and fixed polarization $\delta_{K/K^+}$). These rings all exist by absolute irreducibility of the respective residual representations, and by the discussion in \cite[\S 4.2]{bg:Gdef} (\cite[Lemma 3.4.1]{bg:Gdef} plays a key role here).
By construction, there are natural $\mc O$-algebra maps
\[
R^{\mr{univ}}_{\mr{GL}} \to R^{\mr{univ}}_{\mr{SO}} \to R^{\mr{univ}}_{\mr{GSpin}}.
\]

\begin{lem}\label{defcompare}
The map $R^{\mr{univ}}_{\mr{GL}} \to R^{\mr{univ}}_{\mr{SO}}$ is surjective, and the map $R^{\mr{univ}}_{\mr{SO}} \to R^{\mr{univ}}_{\mr{GSpin}}$ is an isomorphism.
\end{lem}
\begin{proof}
The tangent space of $R^{\mr{univ}}_{\mr{GL}}$ is a subspace of $H^1(\Gamma_{K^+,S}, \rho(\bar{r})(\mf{gl}_{2n+1}))$ and
the tangent space of $R^{\mr{univ}}_{\mr{SO}}$ is a subspace of $H^1(\Gamma_{K^+,S}, P(\bar r)(\mf{so}_{2n+1}))$. For $p \geq 2n+4$, $\mf{gl}_{2n+1}$ is by \cite[Proposition 2]{serre:ss} a semisimple $\mr{SO}_{2n+1}$-module, so \textit{a fortiori} $P(\bar r)(\mf{so}_{2n+1})$ is a $\Gamma_{K^+, S}$-direct summand of $P(\bar{r})(\mf{gl}_{2n+1})$. It follows that the natural map $H^1(\Gamma_{K^+,S}, P(\bar r)(\mf{so}_{2n+1})) \to H^1(\Gamma_{K^+,S}, \rho(\bar{r})(\mf{gl}_{2n+1}))$ is injective; the dual map is surjective, and we conclude by Nakayama's lemma that the map on universal deformation rings without local conditions is surjective. It then follows immediately that $R^{\mr{univ}}_{\mr{GL}} \to R^{\mr{univ}}_{\mr{SO}}$ is surjective as well. 

The canonical map $R^{\mr{univ}}_{\mr{SO}} \to R^{\mr{univ}}_{\mr{GSpin}}$ is induced by a corresponding map on global deformation rings without local conditions; the latter is an isomorphism because it induces isomorphisms $H^i(\Gamma_{K^+, S}, \bar{r}(\mf{so}_{2n+1})) \xrightarrow{\sim} H^i(\Gamma_{K^+, S}, P(\bar{r})(\mf{so}_{2n+1}))$ (one of these copies of $\mf{so}_{2n+1}$ is the derived Lie algebra of $\mr{GSpin}_{2n+1}$, while the other is the Lie algebra of $\mr{SO}_{2n+1}$; the natural map between them is an isomorphism since $p$ is not $2$), which for $i=1$ and $i=2$ means that it identifies both the tangent spaces and obstruction spaces for the two deformation functors. The local lifting rings without local conditions are similarly isomorphic, and this implies that the specified components of the local lifting rings 
are isomorphic. That $R^{\mr{univ}}_{\mr{SO}} \to R^{\mr{univ}}_{\mr{GSpin}}$ is an isomorphism now follows.
\end{proof}

We now complete the proof of Theorem \ref{lifting thm}:
\begin{proof}[Proof of Theorem \ref{lifting thm}]
By our assumptions on $\bar r$, the representation $\rho(\bar{r})$ (strictly speaking, the pair $(\op{std}(\bar{r})|_{\Gamma_K}, \delta_{K/K^+})$ with its prolongation $\rho(\bar{r})$) satisfies the assumptions of \cite[Theorem 4.2.11]{calegari-emerton-gee}, so $(\op{std}(\bar{r})|_{\Gamma_K}, \delta_{K/K^+})$ lifts to a potentially automorphic compatible system $(\{s_{\lambda}\}_{\lambda}, \delta_{K/K^+})$ (here $\lambda$ is indexed over places of a suitable number field, or over field isomorphisms $\cmplx \xrightarrow{\sim} \bQl$). The method of proof of \textit{loc. cit.} implies that the deformation ring $R^{\mr{univ}}_{\mr{GL}}$ is $\mc{O}$-finite, since that is the means by which lifts are shown to exist. Not to impose on the reader to have to inspect the proof, we can simply take the conclusion of \textit{loc. cit.} and then invoke \cite[Lemma 1.4.28]{calegari-emerton-gee}.

Lemma \ref{defcompare} then implies that $R^{\mr{univ}}_{\mr{GSpin}}$ is $\mc O$-finite. We claim that $R^{\mr{univ}}_{\mr{GSpin}}$ has Krull dimension at least one. Indeed, this follows from \cite[Theorem B]{bg:Gdef}: the assumptions there are satisfied since we have assumed that $\bar r$ is odd, $\op{std}(\bar r)|_{K(\mu_p)}$ is irreducible (which under our assumption on $p$ implies that $H^0(\Gamma_{K^+},\bar r(\mf{so}_{2n+1})(1))$ vanishes), and the Hodge-Tate cocharacters are regular. Thus $R^{\mr{univ}}_{\mr{GSpin}}$ has a $\bQp$-point, which proves our theorem.  
\end{proof}

\section{Potential automorphy of mod $p$ and $p$-adic representations} \label{sec:automorphy}
In this section we will prove two potential automorphy theorems, one for $\mr{GSpin}_{2n+1}(\ov{\mb{F}}_p)$-valued representations, and one for $\mr{GSpin}_{2n+1}(\bZp)$-valued representations. For the first we will begin by applying a version of Theorem \ref{lifting thm}, but in fact we can and do get away with a somewhat simpler result, relying on \cite{blggt} rather than the refinement in \cite{calegari-emerton-gee}. Our $p$-adic potential automorphy theorem will however require \cite{calegari-emerton-gee} for its most general statement.

\subsection{The mod $p$ case} Let $F^+$ be a totally real field, and let $\bar r: \Gamma_{F^+} \to \gspin{2n+1}{\bFp}$ be a continuous mod $p$ representation satisfying the following slight relaxation of the conditions (1), (2), and (3) of Hypothesis \ref{hypotheses} (we replace $K^+$ by $F^+$ since eventually we will apply the result of \S \ref{sec:lifting} to some finite extension $K^+$ of this $F^+$):
\begin{hyp}\label{modppotauthyp}
Assume that $\bar{r}$ satisfies the following:
\begin{enumerate}
    \item $\op{std}(\bar{r})|_{\Gamma_{F^+(\mu_p)}}$ is absolutely irreducible.
    \item $p \geq 2n+4$.\footnote{The potential automorphy theorem \cite[Theorem 4.3.1]{blggt} we rely on needs the following hypothesis, somewhat weaker than the assumption in \cite{calegari-emerton-gee}: let $d$ be the maximal dimension of an irreducible subrepresentation of the restriction of $\op{std}(\bar{r})$ to the closed subgroup of $\Gamma_{F^+}$ generated by all Sylow pro-$p$-subgroups, and assume that $p \geq 2(d+1)$. In Lemma \ref{defcompare} above we use the slightly stricter requirement $p \geq 2n+4$.}
    \item $\bar{r}$ is odd.
\end{enumerate}
\end{hyp}
In essence, we retain the irreducibility and oddness hypotheses from \S \ref{sec:lifting}, but not those on the existence of sufficiently nice local lifts. In this section, we use the results in \cite{blggt} and \cite{ks} to prove that $\bar r$ is potentially automorphic in the sense that over a finite totally real extension of $F^+$, $\bar r$ is the mod $p$ reduction of the Galois representation attached in \cite{ks} to a suitable automorphic representation of $\GSp{2n}$. 

 Let $S$ be a finite set of finite places of $F^{+}$ including the places above $p$, the places where $\bar r$ is ramified, and one other auxiliary place $v_{\mr{St}}$ that is not above $p$. 
 Choose a CM quadratic extension $F$ of $F^+$ that is linearly disjoint from $F^+(\bar{r}, \mu_p)$, and in which all primes in $S$ split. For each $v\in S$, let $E_v/F^+_v$ be a finite extension over which $\bar r|_v$ becomes trivial. We may and do also assume that when $v=v_{\mr{St}}$ the residue field of $E_v$ has order $q^f$ where $f$ is even (this is inessential but convenient). By \cite[Lemma 4.1.2]{cht}, there is a finite totally real extension $K^+/F^+$ linearly disjoint from $F(\bar r, \mu_p)$ such that for each place $v$ of $S$ and each place $w$ of $K^+$ above $v$, the extension $E_v/F^+_v$ is isomorphic to $K^+_w/F^+_v$. Set $K=F K^+$. Continuing to write $S$ for the finite places of $K^+$ above $S$, we have the following result:
\begin{prop}\label{lift}
There is a choice of lift $\theta \colon \Gal{K^+} \to \bZp^\times$ of $N \circ \bar{r}|_{\Gal{K^+}}$ 
and a lift $r \colon \Gal{K^+} \to \mr{GSpin}_{2n+1}(\bZp)$ of $\bar{r}|_{\Gal{K^+}}$ with $N \circ r= \theta$ such that
\begin{itemize}
\item For all places $v \vert v_{\mr{St}}$ (which recall does not lie above $p$), the point corresponding to $r|_{\Gal{K^+_v}}$ lies on the same connected component of $R^{\Box, \theta}_{\bar{r}|_v}$ as a representation whose associated Weil-Deligne representation is an unramified twist of the local Langlands parameter of the Steinberg representation.\footnote{By this we just mean the principal parameter in the sense of \cite[\S 6]{gr97}; it is the expected local Langlands parameter of the Steinberg representation.}
\item For all $v \vert p$, the restriction $r|_{\Gal{K^+_v}}$ is potentially crystalline with 
regular Hodge-Tate cocharacters, and $\op{std}(r|_{\Gal{K^+_v}})$ is potentially diagonalizable. 
\end{itemize}
\end{prop}
\begin{proof}
The hypotheses of Theorem \ref{lifting thm} are clear except for the local conditions at finite primes. Note that by the choice of $K^+$, $\bar{r}|_{\Gal{K^+}}$ is everywhere unramified and is trivial at the places in $S$. Consider a place $v \vert p$ of $K^+$. Choose a set of integers $\{s_1, s_2, \ldots, s_n\}$ such that for all $i \neq j$, we have $s_i \neq \pm s_j$, and all $s_i$ are non-zero.
Let $T$ be a maximal torus of $\mr{GSpin}_{2n+1}$ (say, the preimage of the diagonal torus in $\mr{SO}_{2n+1}$), so that we can write a based root datum for $\mr{GSpin}$ (with respect to $T$ and a suitable choice of Borel; see, e.g., \cite[\S 2.8]{pat12}) as
\begin{align*}
X^\bullet(T)&= \bigoplus_{i=1}^n \ints \chi_i \oplus \ints(\chi_0+ \frac{\sum_{i=1}^n \chi_i}{2}) \subset \bigoplus_{i=0}^n \rats \chi_i \\
\Delta^\bullet&= \{\chi_i-\chi_{i+1}\}_{i=1}^{n-1} \cup \{\chi_n\} \\
X_{\bullet}(T)&= \bigoplus_{i=1}^n \ints(\lambda_i+ \frac{\lambda_0}{2}) \oplus \ints \lambda_0 \\
\Delta_{\bullet}&= \{\lambda_i-\lambda_{i+1}\}_{i=1}^{n-1} \cup \{2 \lambda_n\},
\end{align*}
where the pairing is given by $\langle \chi_i, \lambda_j \rangle = \delta_{i, j}$ for all $i, j$.

Now define a $T(\Zp)$-valued lift $r_v \colon \Gal{K^+_v} \to \mr{GSpin}_{2n+1}(\Zp)$ of $\bar{r}|_v$ by
\[
r_v= \prod_{i=1}^n (\lambda_i+ \frac{\lambda_0}{2})(\kappa^{s_i}\cdot [\bkp]^{-s_i})
\]
where $[\bkp]$ denotes the Teichm\"{u}ller lift of $\bkp$ (as $v \vert p$ varies, we may take the same integers $s_i$; but what is essential is that $s=\sum_{i=1}^n s_i$ be independent of $v \vert p$). Then we see easily that $r_v$ lifts $\bar{r}_v$, is potentially crystalline with regular Hodge-Tate cocharacters, $\op{std}(r_v)$ is potentially diagonalizable, and $N(r_v)= (\kappa \cdot [\bar{\kappa}]^{-1})^{\sum_{i=1}^n s_i}$. Set $s= \sum_{i=1}^n s_i$. We define $\theta \colon \Gamma_{K^+} \to \mc{O}^\times$ by $\theta= \kappa^s \cdot [\kappa]^{-s} \cdot [N(\bar{r})]$. In particular, $N(r_v)=\theta|_{\Gamma_{K^+_v}}$ for all $v \vert p$.

Now for each $v \vert v_{\mr{St}}$ we define a lift $r_v$ of $\bar{r}|_{\Gal{K^+_v}}$
(which is trivial) with the desired properties, including now that the norm be the character $\theta$ just defined. Namely, let $\sigma_{v}$ be an element of $\gal{K^{+, \mr{tame}}_{v}}{K^+_{v}}$ that maps down to the arithmetic Frobenius in $\gal{K^{+, \mr{unr}}_{v}}{K^+_{v}}$, and let $\tau_{v}$ be a topological generator of the pro-cyclic group $\gal{K^{+, \mr{tame}}_{v}}{K^{+, \mr{unr}}_{v}}$ (so we have $\sigma_{v}\tau_{v}\sigma_{v}^{-1}=\tau_{v}^{q^f}$). Define $\rho_{v}: \gal{K^{+, tame}_{v}}{K^+_{v}} \to \mr{SL}_2(\mc{O})$ by $\rho_{v}(\sigma_{v})= \begin{pmatrix} q^{f/2} & 0 \\ 0 & q^{-f/2} \end{pmatrix}$ (recall that $f$ is even), and $\rho_{v}(\tau_{v})=\begin{pmatrix} 1 & 1 \\ 0 & 1 \end{pmatrix}$. We initially set $r_{v, 0}= \varphi \circ \rho_v$, where $\varphi \colon \mr{SL}_2 \to \mr{Spin}_{2n+1}$ is a principal $\mr{SL}_2$ (for instance, with respect to a choice of pinning for our based root datum for $\mr{GSpin}_{2n+1}$). Then $r_{v, 0}$ is a Steinberg parameter with $N(r_{v,0})=1$, and it remains to introduce a twist to get the correct Clifford norm. Since $\theta|_{\Gamma_{K^+_v}}$ is unramified with 
\[
\theta(\frob{v})= q_v^s [q_v^{-s}] \equiv 1 \pmod p,
\]
there exists an unramified character $\alpha_v \colon \Gamma_{K^+_v} \to \Zp^\times$ such that $\alpha_v(\frob{v})^2= \theta(\frob{v})$. We then define $r_v$ to be $r_{v, 0} \cdot \alpha_v$, so that now $N(r_v)= N(r_{v, 0})\cdot \alpha_v^2= \theta|_{\Gamma_{K^+_v}}$. A straightforward computation shows that $r_v$ is robustly smooth in the sense of \cite{bg:Gdef} (or ``very smooth" in the sense of \cite{blggt}).

The lifts $r_v$ we have just defined for $v \vert p$ and $v \vert v_{\mr{St}}$ satisfy the desired local conditions and have Clifford norm equal to the restriction of $\theta$. For other places $v \in S$, recalling that $\bar{r}|_{\Gamma_{K^+_v}}$ is trivial, we can simply let $r_v$ be an unramified twist (as in the Steinberg case, to get the correct central character) of the trivial representation.\footnote{One could certainly refine this argument to achieve more control over our eventual lifts at places in $S$ away from $p$ and $v_{\mr{St}}$, but we have no need to do so.} Thus, 
we can apply Theorem \ref{lifting thm} (and in fact its simpler variant based on Remark \ref{blggtlift}) to deduce the existence of a lift $r \colon \Gal{K^+, S} \to \mr{GSpin}_{2n+1}(\bZp)$ of $\bar{r}|_{\Gal{K^+}}$ such that for all $v \in S$, $r|_{\Gal{K^+_v}}$ lies on the same irreducible component of $\mr{Spec}(R^{\Box, \theta}_{\bar{r}|_v} \otimes \bQp)$ as $r_v$, and for $v \vert p$ also has the same $p$-adic Hodge type as $r_v$. The proof of the proposition is complete.
\end{proof}

For each $v \vert p$ and each embedding $\tau \colon K^+_v \to \bQp$, we let $\mu(r|_v, \tau)=\mu(r_v, \tau) \colon \mb{G}_m \to \mr{GSpin}_{2n+1}$ be the Hodge-Tate cocharacter (properly, a conjugacy class of cocharacters) of $r|_v$, or equivalently of $r_v$.

We will now apply the potential automorphy theorem of \cite{blggt} to the representation $\op{std}(r): \Gamma_{K^+} \to \gl{2n+1}{\bQp}$. We begin with a remark on the normalization of the construction of automorphic Galois representations, which is summarized in \cite[Theorem 2.1.1]{blggt} and proven by work of many people, with the result we need contained in \cite[Theorem 1.2, Remark 7.6]{sh11}. These references imply that for every regular L-algebraic self-dual cuspidal automorphic representation $\Pi$ of $\mr{GL}_{2n+1}(\mb{A}_{K^+})$, there is an associated $r_{\Pi, \iota} \colon \Gamma_{K^+} \to \mr{GL}_{2n+1}(\bQp)$ such that for all $v$ not above $p$, the Weil-Deligne representation associated to $r_{\Pi, \iota}|_{\Gamma_{K^+_v}}$ is, up to Frobenius semisimplification, isomorphic to the local Langlands parameter (in the standard unitary normalization, not the ``arithmetic" normalization) of $\Pi_v$. Moreover, for $v \vert p$, $r_{\Pi, \iota}|_{\Gamma_{K^+_v}}$ is de Rham (and crystalline if $\Pi_v$ is unramified) with labeled Hodge-Tate weights read off from $\Pi_{\infty}$ according to the usual recipe, which we will now outline. If $M$ is a number field, $G$ is a connected reductive group over $M$, and $\pi$ is an automorphic representation of $G(\mathbb{A}_M)$, then for each choice of field embedding $\epsilon \colon M \to \cmplx$, inducing an infinite place of $M$ (that we will also denote by $\epsilon$) and an identification $\overline{M}_{\epsilon}= \cmplx$, we consider the restriction $\phi_{\pi_\epsilon} \colon W_{\overline{M}_\epsilon} \to {}^L G(\cmplx)$ of the archimedean L-parameter. 
Precisely when $\pi_\epsilon$ is L-algebraic, this homomorphism is induced by a pair of algebraic cocharacters, $z \mapsto z^{\mu(\pi_{\epsilon},\epsilon)} \bar{z}^{\nu(\pi_{\epsilon}, \epsilon)}$ of a maximal torus inside the dual group $G^\vee$, and we thus associate to the pair $(\pi, \epsilon)$ the $G^\vee$-conjugacy class of cocharacters $\mu(\pi_\epsilon, \epsilon)$. A conjectural Galois representation $r_{\pi, \iota}$ associated to $\pi$ should then have the property that for any field embedding $\tau \colon M \to \bQp$, which induces $M_v \to \bQp$ for some $v \vert p$, and which via $\iota$ induces an embedding $\epsilon= \iota^{-1} \tau \colon M \to \cmplx$, the Hodge-Tate cocharacter $\mu(r_{\pi, \iota}|_v, \tau)$ and $\mu(\pi_{\epsilon}, \epsilon)$ define the same $G^\vee$-conjugacy class of cocharacters. We will use these notions and notation in what follows in the cases $G= \mr{Sp}_{2n}, \mr{GSp}_{2n}, \mr{GL}_{2n+1}$.

\begin{prop}\label{big pi}
There exist a totally real extension $L^+/K^+$ and a regular L-algebraic self-dual cuspidal automorphic representation $\Pi$ of $\gl{2n+1}{\mb A_{L^+}}$ such that 
$\op{std}(r)|_{\Gal{L^+}} \cong r_{\Pi,\iota}$. Moreover, the automorphic representation $\Pi$ satisfies the following properties:
\begin{enumerate}
\item For each embedding $\epsilon: L^+ \to \cmplx$,
inducing an embedding $\tau:= \iota \epsilon \colon L^+ \to \bQp$ together with a place $v \vert p$ of $L^+$, the cocharacter $\mu(\Pi_{\epsilon}, \epsilon)$ coming from the archimedean L-parameter $\phi_{\Pi_{\epsilon}}$ is conjugate to the Hodge-Tate cocharacter $\mu(\op{std}(r)|_v,\tau)$.
\item $\Pi_{w}$ is isomorphic to a 
twist of the Steinberg representation of $\gl{2n+1}{L^+_{w}}$ for any finite $L^+$-place $w \vert v_{\mr{St}}$.  
\end{enumerate}
\end{prop}

\begin{proof}
The existence of the field $L^+$ and the cuspidal automorphic representation $\Pi$ follow from \cite[Theorem C]{blggt}, where the assumptions are clearly satisfied by $\op{std}(r)$. The first enumerated item is part of \cite[Theorem 1.2]{sh11}. For the second item,
since $\Pi$ is cuspidal, $\Pi_w$ is generic for all places $w$, and therefore local-global compatibility (\cite[Theorem 1.2(i)]{sh11}) and \cite[Lemma 1.3.2]{blggt} imply that $\op{std}(r)|_{\Gal{L^+_w}}$ is a smooth point on its (generic fiber) local lifting ring. Since $r|_{\Gal{L^+_w}}$ lies on the same irreducible component as the restriction $r_v|_{\Gal{L^+_w}}$ of the Steinberg-type lift constructed in Proposition \ref{lift}, so do $\op{std}(r)|_{\Gal{L^+_w}}$ and $\op{std}(r_v)|_{\Gal{L^+_w}}$ (since the canonical map from the $\mr{GSpin}_{2n+1}$ to the $\mr{GL}_{2n+1}$ deformation ring of course preserves irreducible components). The latter is visibly a smooth point, so by \cite[Lemma 1.3.4 (2)]{blggt} (due to Choi) the inertial restrictions $\op{std}(r)|_{I_{L^+_w}}$ and $\op{std}(r_v)|_{I_{L^+_w}}$ are isomorphic.
It follows from the local Langlands correspondence for $\GL{2n+1}(L^+_w)$ that $\Pi_{w}$ is an unramified twist of the Steinberg representation of $\gl{2n+1}{L_{w}}$.
\end{proof}

\begin{rmk}\label{splitplaces}
The extension $L^+/K^+$ may be chosen with the following properties: for any fixed finite Galois extension $K^{\mr{avoid}}/K^+$, and any finite set $U$ of places of $K^+$, the extension $L^+$ can be taken to be linearly disjoint from $K^{\mr{avoid}}$ over $K^+$, and such that the primes in $U$ all split completely in $L^+/K^+$. The former assertion is part of \cite[Theorem C]{blggt}, whereas the latter requires a small additional argument and is explained in the last paragraph of the proof of \cite[Theorem 2.1.16]{calegari-emerton-gee}. 
\end{rmk}
In the next result we use Arthur's endoscopic classification, which for our purposes proves the local Langlands correspondence for $\mr{Sp}_{2n}(L^+_w)$ (for all places $w$; of course, for infinite places this is due to Langlands) and describes the discrete automorphic spectrum of $\mr{Sp}_{2n}(\mb{A}_{L^+})$ in terms of self-dual isobaric automorphic representations of $\mr{GL}_{2n+1}(\mb{A}_{L^+})$. We will try to phrase the following argument in a way that will help the uninitiated reader to parse the main results of \cite{art13}. In what follows, for an admissible smooth representation (or admissible $(\mf{g}, K)$-module) $\tau$ of either $\mr{GL}_{2n+1}(L^+_v)$ or $\mr{Sp}_{2n}(L^+_v)$, we will write $\phi_{\tau}$ for the local Langlands parameter associated to $\tau$ in, respectively, \cite{ht} or \cite{art13}.

\begin{prop}\label{descent}
There is a cuspidal automorphic representation $\pi$ of $\sp{2n}{\mb A_{L^+}}$ with the following properties:
\begin{enumerate}
\item For every place $v$ of $L^+$, 
\[
\phi_{\Pi_v}\cong \op{std} \circ \phi_{\pi_v}.
\]
\item For each embedding $\epsilon: L^+ \to \cmplx$ 
inducing as before an embedding $\tau: L^+ \to \bQp$ together with a place $v|p$ of $L^+$, the cocharacters $\mu(\pi_\epsilon, \epsilon)$ and
$\mu(P(r)|_v,\tau)$ are conjugate in $\mr{SO}_{2n+1}$. Moreover, $\pi_{\epsilon}$ is a discrete series representation of $\mr{Sp}_{2n}(L^+_{\epsilon})$, with infinitesimal character $\mu(\pi_{\epsilon}, \epsilon)$.
\item For any finite $L^+$-place $w|v_{\mr{St}}$, the local component $\pi_{w}$ is isomorphic to the Steinberg representation of $\sp{2n}{L^+_{w}}$.  
\end{enumerate}
\end{prop}
\begin{proof}
Since $\Pi$ is self-dual cuspidal automorphic, it corresponds to a simple generic parameter in the sense of \cite{art13}.
By \cite[Theorem 1.4.1]{art13} there is a unique elliptic (and in fact simple) endsocopic group $G$ with a discrete automorphic representation $\pi$ of $G$ whose system of unramified Hecke eigenvalues gives rise to those of $\Pi$; and since $2n+1$ is odd, $G$ is necessarily the split group $\mr{Sp}_{2n}/L^+$. Simply by definition, $\Pi$ thus gives rise to a parameter $\psi \in \widetilde{\Psi}(G)$, in the sense of \cite[\S 1.4]{art13}. By \cite[Theorem 1.4.2]{art13}, the local Langlands parameters $\phi_{\Pi_v}$ factor through $\mr{SO}_{2n+1} \subset \mr{GL}_{2n+1}$, and then the central results \cite[Theorem 1.5.1, Theorem 1.5.2]{art13} establish the local Langlands correspondence for $G$ and show that the descent $\pi$ of $\Pi$ (the multiplicity here is one, in fact) is everywhere locally compatible with this local Langlands correspondence, which yields part (1) of our Proposition.

The relation between $\phi_{\pi_{\epsilon}}|_{W_{\overline{L}_{\epsilon}^+}}$ and $\mu(P(r)|_v, \tau)$ in (2) follows from (1), part (1) of Proposition \ref{big pi}, and the fact that $\mr{SO}_{2n+1}$-valued cocharacters are $\mr{SO}_{2n+1}$-conjugate if they are $\mr{GL}_{2n+1}$-conjugate.  The fact that $\pi_{\epsilon}$ is in the discrete series with the claimed infinitesimal character is then a standard consequence, since there are no other possibilities for the extension from $W_{\cmplx}$ to $W_{\reals}$ of this (regular and tempered) parameter.

Now we consider a place $w \vert v_{\mr{St}}$ and show (3). As $\Pi_{w}$ is a 
twist of a Steinberg representation, and the parameter $\phi_{\Pi_w}$ takes values in $\mr{SO}_{2n+1}$, $\Pi_w$ must itself be the Steinberg representation: self-duality forces it to be at most a quadratic 
twist of Steinberg, and then trivial determinant forces the twist to be trivial. Thus the parameter $\phi_{\pi_w}$ has $\op{std}(\phi_{\pi_w})$ equal to the Steinberg parameter for $\mr{GL}_{2n+1}$. It was shown in \cite[Proposition 8.2]{ms} that the Steinberg representation $\mr{St}_w$ of $\mr{Sp}_{2n}(L^+_w)$ has this same parameter. Thus $\pi_w$ and $\mr{St}_w$ lie in the same local L-packet $\widetilde{\Pi}_{\phi_{\pi_w}}$ (in the notation of \cite[Theorem 1.5.1]{art13}). This parameter is bounded, and Arthur shows (\cite[Theorem 1.5.1(b)]{art13}) that the packet is thus in bijection with the group denoted $\mc{S}_{\phi_{\pi_w}}$ (the component group of the centralizer of the parameter). This group is visibly trivial, so the L-packet contains a single element, and therefore $\pi_w \cong \mr{St}_w$.
\end{proof}
The following argument is complicated slightly by the fact that we do not know at the outset the expectation (\cite[III.10.3]{bor}) of how to compute the central character $\omega_{\pi}$ from the local L-parameters of $\pi$ (except at archimedean and unramified primes).
\begin{prop}\label{lift aut}
There is a cuspidal automorphic representation $\tilde\pi$ of $\gsp{2n}{\mb A_L^+}$ satisfying the following properties:
\begin{enumerate}
\item For any place $v$ of $L^+$ that is either archimedean or for which $\tilde{\pi}_v$ is unramified (in these two cases the corresponding L-parameters $\phi_{\tilde{\pi}_v}$ are defined), $P\circ \phi_{\tilde\pi_v}=\phi_{\pi_v}$.
\item The representation $\tilde{\pi}$ is L-cohomological in the sense of \cite{ks}. More precisely, for each embedding $\epsilon: L^+ \to \cmplx$, which identifies the algebraic closure of $L^+_{\epsilon}$ with $\cmplx$ and induces (via $\iota$) an embedding $\tau: L^+ \to \bQp$ together with a place $v|p$ of $L^+$, $\mu(\tilde{\pi}_{\epsilon}, \epsilon)$ equals $\mu(r|_v, \tau)$.
\item For any place $w \vert v_{\mr{St}}$ of $L^+$, the local component $\tilde\pi_{w}$ is a twist of the Steinberg representation of $\gsp{2n}{L^+_{w}}$.
\end{enumerate}
\end{prop}
\begin{proof}
Let $\tilde{\omega}_0$ be a finite-order Hecke character of $L^+$ extending the central character $\omega_{\pi}$, and set $\tilde{\omega}= \tilde{\omega}_0 \cdot | \cdot |^s$, where $s= \sum_{i=1}^n s_i$ is the integer appearing in the definition of $\theta=N(r)$ in Proposition \ref{lift}. Then $\tilde{\omega}$ clearly also extends $\omega_{\pi}$. By the argument of Langlands-Labesse (for $\mr{SL}_2$ in \cite{ll} and generalized in \cite[Proposition 3.1.4]{pat12} using an argument of Flicker from \cite{flicker}), there is a cuspidal automorphic representation $\tilde{\pi}$ of $\mr{GSp}_{2n}(\mb{A}_{L^+})$ with central character $\omega_{\tilde{\pi}}= \tilde{\omega}$ and extending $\pi$ (i.e., $\pi$ is a constituent of the restriction $\tilde{\pi}|_{\mr{Sp}_{2n}(\mb{A}_{L^+})}$). Part (1) follows from \cite[Corollary 3.1.6]{pat12}. We claim that for all $\epsilon \colon L^+_\epsilon \to \cmplx$, $\tilde{\pi}_\epsilon$ is L-cohomological with the claimed infinitesimal character. Indeed, the only quasi-cocharacter of $T \subset \mr{GSpin}_{2n+1}$ (notation as in Proposition \ref{lift}) lifting $\mu(\pi_{\epsilon}, \epsilon)$ and having norm equal to $z \mapsto z^s$, $s=\sum_{i=1}^n s_i$, is the (genuine) cocharacter denoted (in the root datum notation introduced above) $\sum_{i=1}^n s_i (\lambda_i + \frac{\lambda_0}{2})$, and equal to $\mu(r|_v, \tau)$. This must be the infinitesimal character $\mu(\tilde{\pi}_{\epsilon}, \epsilon)$, so we see that $\tilde{\pi}_{\epsilon}$ is L-algebraic and (because it extends $\pi_{\epsilon}$) essentially discrete series. We will check that $\tilde{\pi}_{\epsilon}$ is moreover ``L-cohomological" in the sense of \cite{ks}, i.e. that $\tilde{\pi}_{\epsilon} \cdot |\lambda_0|^{\frac{n(n+1)}{4}}$ is cohomological in the sense of \cite[Definition 1.12]{ks} ($\lambda_0$ corresponds to the similitude character of $\mr{GSp}_{2n}$). Since $\pi_{\epsilon}$ is discrete series with infinitesimal character $\mu(\pi_{\epsilon}, \epsilon)$, it is $\xi^\vee_{\mu(\pi_{\epsilon}, \epsilon)-\rho}$-cohomological (\cite[Theorem V.3.3]{bw}), where $\xi_{\mu(\pi_{\epsilon}, \epsilon)-\rho}$ is the highest-weight $\mu(\pi_{\epsilon}, \epsilon)-\rho$ representation of $\mr{Sp}_{2n}(\cmplx)$, and $\rho$ is the half-sum of positive roots for the choice of root basis used also to parametrize highest weights. We need an extension $\tilde{\xi}$ of $\xi_{\mu(\pi_{\epsilon}, \epsilon)-\rho}$ to $\mr{GSp}_{2n}(\cmplx)$ such that (letting $K$ be a maximal compact subgroup of $\mr{Sp}_{2n}(\reals)$)
\[
H^*(\mf{gsp}_{2n}(\cmplx), K\cdot Z_{\mr{GSp}_{2n}}(\reals); \tilde{\pi}_{\epsilon}\cdot|\lambda_0|^{\frac{n(n+1)}{4}} \otimes \tilde{\xi}^\vee) \neq 0.
\]
From the definition of this cohomology group, and the fact that $\pi_{\epsilon} \otimes \xi^\vee$ has non-zero $(\mf{sp}_{2n}(\cmplx), K)$-cohomology, this reduces to checking that we can find an extension $\tilde{\xi}^\vee$ of $\xi^\vee_{\mu(\pi_{\epsilon}, \epsilon)-\rho}$ such that $\tilde{\xi}$ and $\tilde{\pi}_{\epsilon}\cdot|\lambda_0|^{\frac{n(n+1)}{4}}$ have the same central character; note that we will only have to check this on $\reals^\times_{>0} \subset Z_{\mr{GSp}_{2n}}(\reals)$ since $-1 \in K$.\footnote{The definition of ``cohomological" used in \cite{ks} for a general real reductive group $H$ uses $(\mr{Lie}(H(\cmplx)), K_H^0 Z_H(\reals))$-cohomology, where $K_H^0$ is the identity component of a maximal compact subgroup $K_H$ of $H(\reals)$. For $H= \mr{GSp}_{2n}$, $K_H^0$ is a maximal compact subgroup of $\mr{Sp}_{2n}(\reals)$.} We choose $\tilde{\xi}$ to be the representation with highest weight $\mu(\tilde{\pi}_{\epsilon}, \epsilon)-\rho+\frac{n(n+1)}{4}\lambda_0 \in X_{\bullet}(T)$ (the $\lambda_0$ term corrects the failure of $\rho$ to be integral for $\mr{GSp}_{2n}$). The claim is now proven, since this highest weight representation has the same infinitesimal character as $\tilde{\pi}_{\epsilon} \cdot |\lambda_0|^{\frac{n(n+1)}{4}}$.

Now we prove (3). By Proposition \ref{descent}, part (3), $\pi_{w}$ is Steinberg for any finite $L^+$-place $w|v_{\mr{St}}$. Consider the representation $\tau_w:=\pi_w \times \tilde{\omega}_w$ of the subgroup $H_w:= \mr{Sp}_{2n}(L^+_w)\cdot \widetilde{Z}(L^+_w) \subset \mr{GSp}_{2n}(L^+_w)$, where $\widetilde{Z}$ denotes the center of $\mr{GSp}_{2n}$ (such a representation makes sense because $\tilde{\omega}$ extends $\omega_{\pi}$). We also let $\widetilde{\mr{St}}_w$ denote the Steinberg representation of $\mr{GSp}_{2n}(L^+_w)$. We want, for a suitable twist $\widetilde{\mr{St}}_w \cdot \alpha$ (i.e., twisting by $\mr{GSp}_{2n}(L^+_w) \xrightarrow{\lambda_0} L_w^{+, \times} \xrightarrow{\alpha} \cmplx^\times$), to compare $\widetilde{\mr{St}}_w \cdot \alpha$ to $\op{Ind}_{H_w}^{\mr{GSp}_{2n}(L^+_w)}(\tau_w)$, for which we will simply write $\op{Ind}(\tau_w)$ in what follows. First we note that the restriction $\widetilde{\mr{St}}_w|_{H_w}$ is isomorphic to $\mr{St}_w$ extended trivially along $\widetilde{Z}(L^+_w)$, as is clear from the definition of the Steinberg representation. Next, since $\tilde{\omega}_w|_{\pm 1}$ is trivial (as $\tilde{\omega}_w$ extends $\omega_{\pi_w}$), there exists a character $\alpha$, in fact $L_w^{+,\times}/(L_w^{+,\times})^2= [\mr{GSp}_{2n}(L^+_w): H_w]$ distinct characters, satisfying $\alpha^2=\tilde{\omega}_w$. Then, for each such $\alpha$, Frobenius reciprocity and Schur's lemma imply that
\[
\op{Hom}_{\mr{GSp}_{2n}(L^+_w)}(\widetilde{\mr{St}}_w \cdot \alpha, \op{Ind}(\tau_w))= \op{Hom}_{H_w}(\tau_w, \tau_w),
\]
is one-dimensional. Again by Frobenius reciprocity $\tilde{\pi}_w$ is a Jordan-H\"{o}lder constituent of the (easily seen to be semisimple) representation $\op{Ind}(\tau_w)$, and since (by definition of the Steinberg representation) all $\mr{GSp}_{2n}(L^+_w)$-conjugates of $\tau_w$ are isomorphic, another application of Frobenius reciprocity shows that $\op{Hom}_{\mr{GSp}_{2n}(L^+_w)}(\op{Ind}(\tau_w), \op{Ind}(\tau_w))$ has dimension $[\mr{GSp}_{2n}(L^+_w): H_w]$. It follows that 
\[
\op{Ind}(\tau_w) \cong \bigoplus_{\alpha^2= \omega_w} \widetilde{\mr{St}}_w \cdot \alpha,
\]
and therefore for some such $\alpha$, $\tilde{\pi}_w$ is isomorphic to $\widetilde{\mr{St}}_w \cdot \alpha$.
\end{proof}
The main result of \cite{ks} now implies the existence of a Galois representation
\[
r_{\tilde\pi}=r_{\tilde\pi, \iota}: \Gamma_L^+ \to \gspin{2n+1}{\bQp}
\]
satisfying most of the expected properties, including that $P(r_{\tilde{\pi}})= P(r)$, $N(r_{\tilde{\pi}})= \tilde{\omega}$ (abusively writing $\tilde{\omega}$ for the Galois character associated to $\tilde{\omega}$ by global class field theory), and $r_{\tilde{\pi}}$ is locally compatible with $\tilde{\pi}$ at almost all (unramified) places. This $r_{\tilde{\pi}}$ is nearly the Galois representation we want, but first we need to relate $N(r)$ to the Hecke character $\tilde{\omega}$:

\begin{cor}\label{twist}
At all places $v$ of $L^+$, the central character $\omega_{\pi_v}$ is computed in terms of $\phi_{\pi_v}$ according to the expected description in \cite[III.10.3]{bor}. The representation $r_{\tilde\pi}$ is conjugate to $r \otimes \delta$ for a continuous Galois character $\delta: \Gamma_L^+ \to Z(\bQp)$, and we can choose $\tilde{\omega}$ in Proposition \ref{lift aut} to be the Hecke character corresponding to $N(r)$ under global class field theory, so that $\delta^2=1$. 
\end{cor}
\begin{proof}
Recall the local Langlands desideratum that if $\phi_{\pi_v}$ is the L-parameter of $\pi_v$, then $\omega_{\pi_v}$ should be computed as follows: lift $\phi_{\pi_v}$ to a $\mr{GSpin}_{2n+1}(\cmplx)$-valued parameter $\widetilde{\phi}_{\pi_v}$, and then restrict
\[
N(\widetilde{\phi}_{\pi_v}) \circ \op{rec}_v^{-1} \colon (L_v^+)^\times \to \cmplx^\times
\]
to $\mu_2(L_v^+)$. The Weil-Deligne representation associated to $r_{\tilde{\pi}}|_{\Gal{L^+_v}}$ (for $v$ above $p$ this must be taken in the sense of Fontaine: see, e.g., \cite{bg:Gdef} for a description for representations valued in general groups) provides such a lift $\widetilde{\phi}_{\pi_v}$,\footnote{At $v$ above $p$, this claim makes use of local-global compatibility of $r_{\Pi, \iota}|_{\Gal{L^+_v}}$. Since $2n+1$ is odd, this is proven in \cite[Theorem A]{blggtens}.} and then (since, by \cite[Theorem A]{ks}, $N(r_{\tilde{\pi}})$ corresponds to $\tilde{\omega}$) $N(\widetilde{\phi}_{\pi_v})$ corresponds to $\tilde{\omega}_v$ under local class field theory; but by construction $\tilde{\omega}_v$ extends $\omega_{\pi_v}$, so the first claim is proven.
Since $r$ and $r_{\tilde{\pi}}$ both lift $P(r)$, they are twists: $r_{\tilde{\pi}} \cong r \otimes \delta$ for some character $\delta \colon \Gal{L^+} \to Z(\bQp)$. As they both have Clifford norm with the same Hodge-Tate weights, $\delta$ is finite-order. It follows, since $N(r_{\tilde{\pi}})= \delta^2 \cdot N(r)$, that the Hecke character corresponding to $N(r)$ also extends $\omega_{\pi}$, hence that we can choose $\tilde{\omega}$, and thus $\tilde{\pi}$, in Proposition \ref{lift aut} such that $\tilde{\omega}$ corresponds to $N(r)$ under global class field theory.
\end{proof}

Choosing $\tilde{\omega}$ corresponding to $N(r)$, constructing $\tilde{\pi}$ as in Proposition \ref{lift aut}, and then if necessary using Corollary \ref{twist} to twist $\tilde{\pi}$ by the Hecke character corresponding to $\delta^{-1}$, we obtain our first main result:
\begin{thm} \label{residual aut}
Let $\bar r$ be as in the beginning of this section, satisfying Hypothesis \ref{modppotauthyp}. Then $\bar r$ is potentially automorphic, i.e., there exists a totally real extension $L^+/F^+$ and a cuspidal automorphic representation $\tilde\pi$ of $\gsp{2n}{\mb A_{L^+}}$ that is L-cohomological and twist of Steinberg at some finite prime, such that a suitable $\mr{GSpin}_{2n+1}$-conjugate of the Galois representation $r_{\tilde\pi}: \Gamma_{L^+} \to \gspin{2n+1}{\bQp}$ constructed in \cite{ks} lifts $\bar r|_{L^+}$.
\end{thm}

\subsection{The $p$-adic case}
Note that Theorem \ref{residual aut} makes no local hypothesis on $\bar{r}$ at finite primes. We can prove an analogue of Theorem \ref{residual aut} for $p$-adic representations if we assume they satisfy a Steinberg-type local condition at some finite prime. The argument given above implies this with little modification:
\begin{thm}\label{padicpotaut}
Let $F^+$ be a totally real field, and let $r \colon \Gamma_{F^+} \to \mr{GSpin}_{2n+1}(\bZp)$ be a continuous representation, unramified outside a finite set $S$ of primes containing all $v\vert p$, having geometric Clifford norm $N(r)$, and satisfying:
\begin{itemize}
    \item $p>2(2n+1)$.
    \item $\bar{r}$ is odd.
    \item For some finite place $v_{\mr{St}}$ of $F^+$, the Frobenius semisimple Weil-Deligne representation associated to $r|_{\Gamma_{F^+_{v_{\mr{St}}}}}$ is equivalent to a twist of the Steinberg-type Langlands parameter (equivalently, $P(r|_{\Gamma_{F^+_{v_{\mr{St}}}}})$ is a Steinberg parameter for $\mr{Sp}_{2n}(F^+_{v_{\mr{St}}})$). 
    \item There exist a quadratic CM extension $F/F^+$ and a character $\mu \colon \Gamma_{F^+} \to \bZp^\times$ such that 
    \begin{itemize}
    \item $F$ does not contain $\zeta_p$, and $\op{std}(\bar{r})|_{\Gamma_{F(\zeta_p)}}$ is irreducible.  
    \item $(\op{std}(r)|_{\Gamma_{F}}, \mu)$ is polarized, and for some (any) choice of prolongation 
    \[
    \rho(\op{std}(\bar{r})|_{\Gamma_{F}}, \mu) \colon \Gamma_{F^+} \to \mc{G}_{2n+1}(\bZp),
    \]
$\rho(\op{std}(r), \mu)|_{\Gamma_{F^+_v}}$ is globally realizable for each $v \vert p$.
    \end{itemize}
\end{itemize}
Then there exist a totally real extension $L^+/F^+$ and a cuspidal automorphic representation $\tilde{\pi}$ of $\mr{GSp}_{2n}(\mathbb{A}_{L^+})$ that is L-cohomological and locally at primes above $v_{\mr{St}}$ isomorphic to twists of the Steinberg representation, such that $r_{\tilde{\pi}, \iota} \cong r|_{\Gamma_{L^+}}$. The extension $L^+$ may be chosen to avoid any fixed finite Galois extension $F^{\mr{avoid}}/F^+$, and to be split at all places above $S$ (and in particular above $v_{\mr{St}}$).
\end{thm}
\begin{rmk}
As previously noted, the local hypothesis at $v \vert p$ is strictly weaker than assuming $\op{std}(r|_{\Gamma_{F^+_v}})$ is regular and potentially diagonalizable. We note that the definition of globally realizable in \cite{calegari-emerton-gee} has already fixed a CM extension, and we have tried to phrase the statement of the theorem in order to allow flexibility in this choice of extension. It is possible, for instance, to replace an initially given $F$ with one that is split at places above $S \setminus \{v \vert p\}$. 
\end{rmk}
\begin{proof}
The polarization condition implies $\op{std}(r)|_{\Gamma_{F}} \cong \op{std}(r)|_{\Gamma_{F}} \otimes \mu|_{\Gamma_{F}}$, since of course $\op{std}(r)|_{\Gamma_{F}}$ arises from $P(r)$. Under the hypothesis on $r|_{\Gamma_{F_{v_{\mr{St}}}^+}}$, this implies $\mu|_{\Gamma_{F}}=1$ (else $\op{std}(r)|_{\Gamma_{F}}$ would be induced from a finite extension). By oddness (which by \cite[Lemma 1.4.4]{calegari-emerton-gee} is automatic), we conclude that $\mu= \delta_{F/F^+}$, and then we may (\cite[Remark 2.1.9]{calegari-emerton-gee}) take our prolongation that is globally realizable at places above $p$ simply to be the familiar composite
\[
\Gamma_{F^+} \xrightarrow{P(r) \times \mr{res}_F} \mr{SO}_{2n+1}(\bZp) \times \mr{Gal}(F/F^+) \to \mc{G}_{2n+1}(\bZp),
\]
the standard prolongation $\rho(r)$ defined in \S \ref{sec:notation}.

The prolongation $\rho(r)$ satisfies the hypotheses of \cite[Corollary 4.2.12]{calegari-emerton-gee}, so we find a CM extension $L/F$ (with totally real subfield $L^+$) and a cuspidal automorphic representation $\Pi_L$ of $\mr{GL}_{2n+1}(\mb{A}_L)$ that is polarized and regular L-algebraic such that $r_{\Pi_L, \iota} \cong \op{std}(r)|_{\Gamma_L}$. Moreover, 
we may assume that $L^+/F^+$ is split at the primes in $S$ and linearly disjoint from the fixed $F^{\mr{avoid}}$: indeed, once we know $\mr{std}(r)|_{\Gamma_F}$ is potentially automorphic, we deduce from \cite[Theorem 5.5.1]{blggt} (the potentially diagonalizable hypothesis there is only to invoke the potential automorphy theorem of \textit{loc. cit.}) that it belongs to a compatible system of odd, regular, weakly irreducible (as in \cite{calegari-emerton-gee}) polarized representations of $\Gamma_F$, and then \cite[Theorem 2.1.16]{calegari-emerton-gee} shows that we can indeed choose $L/F$ (and $\Pi_L$) such that $L^+/F^+$ is split at all primes above $S$.\footnote{Here we are feeding the conclusion of \cite[Corollary 4.2.12]{calegari-emerton-gee} back into one of its essential ingredients in order to obtain a slightly stronger result.}
Next, we note that $\Pi_L^c \cong \Pi_L$ (where $c$ is complex conjugation), so $\Pi_L$ descends to a (regular L-algebraic) cuspidal automomrphic representation $\Pi$ of $\mr{GL}_{2n+1}(\mb{A}_{L^+})$; it has an associated Galois representation that restricts to the irreducible $\op{std}(r)|_{\Gamma_L}$, so twisting if necessary we may in fact assume $r_{\Pi, \iota} \cong \op{std}(r)|_{\Gamma_{L^+}}$. We take this $\Pi$ in Proposition \ref{big pi}, and from this point on the proof of our theorem is identical to the proof of Theorem \ref{residual aut}; note that both $N(r)$ and $\op{std}(r)$ being geometric suffices to imply that $r$ is geometric: this follows from work of Wintenberger and Conrad (see \cite[Theorem 6.2]{conrad}). 
\end{proof}

\section{Solvable Descent and Compatible Systems} \label{sec:compatible}
Potential automorphy theorems for the group $\mr{GL}_N$ imply, by Taylor's Brauer induction argument (\cite[\S 5.3.3]{taylor:galoisreps}), that one can often put a single $p$-adic representation $\Gamma_{F^+} \to \mr{GL}_N(\bQp)$ into a compatible system (over this exact field $F^+$, not merely over the extension $L^+$ where it is shown to be automorphic). In this section we will give a variant of Taylor's argument that applies to a single $\mr{GSpin}_{2n+1}(\bQp)$-valued representation. 
We postpone the proofs of the necessary solvable descent result to Theorem \ref{solvabledescent}, and we first explain how to apply them and our potential automorphy theorem to put certain $\mr{GSpin}_{2n+1}(\bQp)$-valued Galois representations in compatible systems:
\begin{thm}\label{compsystem}
Let $r \colon \Gamma_{F^+} \to \mr{GSpin}_{2n+1}(\bZp)$ satisfy the hypotheses of Theorem \ref{padicpotaut}. Then for all primes $\ell$ and choices of isomorphism $\iota_\ell \colon \cmplx \xrightarrow{\sim} \bQl$ there is a continuous representation $r_{\iota_{\ell}} \colon \Gamma_{F^+} \to \mr{GSpin}_{2n+1}(\bQl)$ such that:
\begin{itemize}
    \item For all but finitely many primes $v$ at which $r$ is unramified, the semisimple conjugacy class of $\iota^{-1}r(\frob{v})^{\mr{ss}}$ agrees with that of $\iota_{\ell}^{-1}r_{\iota_{\ell}}(\frob{v})^{\mr{ss}}$.
    \item For all primes $v \vert \ell$, $r_{\iota_{\ell}}$ is de Rham, and its Hodge-Tate cocharacters are determined up to conjugacy by those of $r$: for all embeddings $\tau \colon F^+ \to \bQl$, determining a place $v \vert \ell$ of $F^+$, and inducing $\iota \iota_{\ell}^{-1} \tau \colon F^+ \to \bQp$ and a corresponding place $\iota_{\ell}^*(v) \vert p$ of $F^+$, the conjugacy classes of $\iota_{\ell}^{-1}\mu(r_{\iota_{\ell}}|_{\Gamma_{F^+_v}}, \tau)$ and $\iota^{-1}\mu(r|_{\Gamma_{F^+_{\iota_{\ell}^*(v)}}}, \iota \iota_{\ell}^{-1} \tau)$ coincide.
\end{itemize}
\end{thm}
\begin{proof}
We write $G_r$ for the Zariski closure of the image of $r$, and we write $G_r^0$ for the identity component of $G_r$ (and we use analogous notation for other Galois representations). Let $F^{+, 0}$ be the fixed field of the preimage $r^{-1}(G_r^0(\bQp))$. By Theorem \ref{padicpotaut}, there exist 
\begin{itemize}
    \item a totally real field $L^+$, which we may assume linearly disjoint from the composite of $F^{+, 0}(\mu_p)$ and the fixed field of $\bar{r}$, and split at all places above $v_{\mr{St}}$;
    \item and a cuspidal automorphic representation $\tilde{\pi}$ of $\mr{GSp}_{2n}(\mathbb{A}_{L^+})$ that is L-cohomological, and at all places $v$ above $v_{\mr{St}}$ is isomorphic to a twist of the Steinberg representation, such that $r_{\tilde{\pi}, \iota} \cong r|_{\Gamma_{L^+}}$.
\end{itemize}
We then have an associated compatible system $\{r_{\tilde{\pi}, \iota_{\ell}}\}_{\iota_{\ell}}$ indexed by field isomorphisms $\iota_{\ell} \colon \bQl \xrightarrow{\sim} \cmplx$. Moreover, $r_{\tilde{\pi}, \iota_{\ell}}$ lifts $r_{\pi, \iota_{\ell}}$, where as previously we denote by $\pi$ a (cuspidal automorphic) constituent of $\tilde{\pi}|_{\mr{Sp}_{2n}(\mathbb{A}_{L^+})}$ and $r_{\pi, \iota_{\ell}} \colon \Gamma_{L^+} \to \mr{SO}_{2n+1}(\bQl)$ the associated Galois representation. The algebraic monodromy groups $G_{r_{\pi, \iota_{\ell}}}$ are independent of $\iota_{\ell}$: each is reductive with a regular unipotent element in the image, hence is irreducible with monodromy group  equal to one of the following: all of $\mr{SO}_{2n+1}$; the image of a principal $\mr{SL}_2$; or $\mr{G}_2$ when $n=3$. Independence of $\ell$ of the formal character (or even just the rank) of maximal tori in the algebraic monodromy groups (a result of Serre; see, e.g., \cite[Proposition 6.12]{larsen-pink:lind}) implies the claim. It follows then that $G_{r_{\tilde{\pi}, \iota_{\ell}}}$ is also independent of $\iota_{\ell}$, since the $r_{\tilde{\pi}, \iota_{\ell}}$ moreover have compatible Clifford norms, and component groups are again by Serre's work independent of $\ell$ (\cite[Proposition 6.14]{larsen-pink:lind}). 

By Theorem \ref{solvabledescent} below, for any intermediate extension $L^+ \supset K^+ \supset F^+$ with $L^+/K^+$ solvable, $r|_{\Gamma_{K^+}}$ is automorphic; more precisely, there exist cuspidal automorphic representations $\tilde{\pi}_{K^+}$ of $\mr{GSp}_{2n}(\mathbb{A}_{K^+})$ (L-cohomological and twist of Steinberg at all places above $v_{\mr{St}}$) such that $r_{\tilde{\pi}_{K^+}, \iota}= r|_{\Gamma_{K^+}}$ (we are justified in writing ``$=$" since the representation constructed by Kret-Shin is unique up to $\mr{GSpin}_{2n+1}$-conjugacy).

Consider the composite $\op{spin}(r) \colon \Gamma_{F^+} \to \mr{GL}_{2^n}(\bQp)$. Writing the trivial representation of $\mr{Gal}(L^+/F^+)$ as a linear combination
\[
1= \sum_j n_j \op{Ind}_{\mr{Gal}(L^+/L_j^+)}^{\mr{Gal}(L^+/F^+)} (\psi_j)
\]
for some intermediate extensions $L_j^+$ such that $\mr{Gal}(L^+/L_j^+)$ is nilpotent for each $j$, some characters $\psi_j \colon \mr{Gal}(L^+/L_j^+) \to \cmplx^\times$, and some (possibly negative) integers $n_j$, we see that as virtual representations 
\[
\op{spin}(r) = \sum_j n_j \op{Ind}_{\Gamma_{L_j^+}}^{\Gamma_{F^+}}(\op{spin}(r|_{\Gamma_{L_j^+}}) \otimes \psi_j).
\]
We now define, for every prime $\ell$ and every choice of isomorphism $\iota_{\ell} \colon \cmplx \xrightarrow{\sim} \overline{\mathbb{Q}}_{\ell}$, the virtual $\ell$-adic representation
\[
R_{\iota_{\ell}}= \sum_j n_j \op{Ind}_{\Gamma_{L_j^+}}^{\Gamma_{F^+}}(\op{spin}(r_{\tilde{\pi}_{L_j^+}, \iota_{\ell}}) \otimes \psi_j).
\]
We have to check three things:
\begin{itemize}
    \item The $\{R_{\iota_{\ell}}\}_{\iota_{\ell}}$ form a (weakly) compatible system of actual representations. 
    \item With the previous point established, each $R_{\iota_{\ell}} \colon \Gamma_{F^+} \to \mr{GL}_{2^n}(\bQl)$ will, up to conjugacy, factor as $\op{spin}(r_{\iota_{\ell}})$ for some $r_{\iota_{\ell}} \colon \Gamma_{F^+} \to \mr{GSpin}_{2n+1}(\bQl)$.
    \item The $r_{\iota_{\ell}}$ form a compatible system of $\mr{GSpin}_{2n+1}$-valued representations, containing $r= r_{\iota}$ (when $\iota_{\ell}= \iota$).
\end{itemize}
Since $\op{spin}(r)$ is not in general irreducible, it seems easiest to give the following \textit{ad hoc} argument, treating the different possible algebraic monodromy groups of $\op{spin}(r)$ separately. 
If $n=3$ and $G_{P(r)}= \mr{G}_2$, the spin representation restricted to $G_r^{\mr{der}}$ is the direct sum of the trivial representation and $\op{std}(r)$; thus it suffices to put $\op{std}(r)$ (and a character) in a strictly compatible system, which can be done by combining the potential automorphy result \cite[Corollary 4.2.12]{calegari-emerton-gee} with the proof of \cite[Theorem 5.5.1]{blggt}, and to check that each member of the compatible system factors through $\mr{G}_2$ and is weakly compatible in the $\mr{G}_2$ sense with $P(r)$ (and  $\mr{Spin}_{2n+1}$). The factorization through $\mr{G}_2$ follows from the proof of \cite[Theorem 6.4]{chenevier:G2}, noting that all members of the 7-dimensional compatible system are Hodge-Tate regular.\footnote{Alternatively, we could avoid the elaborate group theory of \cite{chenevier:G2} and use the full strength of the conclusion of \cite[Theorem 5.5.1]{blggt}, that the compatible system is strictly pure. Then using the fact that $\mr{std}(r)|_{\Gamma_{F^+_{v_{\mr{St}}}}}$ is a Steinberg parameter, the simpler argument of \cite[Corollary 7.3]{ms} suffices.} Thus we have $r_{\iota_{\ell}} \colon \Gamma_{F^+} \to \mr{GSpin}_7(\bQl)$ for all $\iota_{\ell}$, each having the form $r_{\iota_{\ell}}= (\tau_{\iota_{\ell}}, \chi_{\iota_{\ell}})$, where $\tau_{\iota_{\ell}}$ factors through $\mr{G}_2 \subset \mr{Spin}_7$, and $\chi_{\iota_{\ell}}$ is a character valued in the center of $\mr{GSpin}_7$. We know that the $\mr{std}(\tau_{\iota_{\ell}})$ form a $\mr{GL}_7$-compatible system, and that the $\chi_{\iota_{\ell}}$ form a one-dimensional compatible system. To see that $\{r_{\iota_{\ell}}\}_{\iota_{\ell}}$ is a $\mr{GSpin}_7$-compatible system, it suffices by \cite[Lemma 1.3]{ks} to check $\mr{GL}$-compatibility after composition with the three representations $N$, $\mr{std}$, and $\mr{spin}$. $N(r_{\iota_{\ell}})= \chi_{\iota_{\ell}}^2$ is compatible, $\mr{std}(r_{\iota_{\ell}})= \mr{std}(\tau_{\iota_{\ell}})$ is compatible, and $\mr{spin}(r_{\iota_{\ell}})= R_{\iota_{\ell}}$ is compatible, so the case $G_{P(r)}= \mr{G}_2$ is complete.

If $G_{P(r)}= \mr{PGL}_2$, then $\mr{SL}_2 \subset G_r \subset \mr{GL}_2$, and we can simply regard $r$ as the composite $\varphi \circ r_0$ of a principal homomorphism $\varphi \colon \mr{GL}_2 \to \mr{GSpin}_{2n+1}$ (a principal $\mr{SL}_2$ extended to the identity map between the centers) and some $r_0 \colon \Gamma_{F^+} \to \mr{GL}_2(\bQp)$. There is a compatible system $r_{0, \iota_{\ell}}$ containing $r_0$, and then we can construct the desired $\mr{GSpin}_{2n+1}$-valued compatible system as $\varphi \circ r_{0, \iota_{\ell}}$.

We now consider the remaining case, where $G_r$ contains $\mr{Spin}_{2n+1}$. Merely for technical convenience, we replace $r$ by its twist by some power of the cyclotomic character, ensuring that $G_r= \mr{Gspin}_{2n+1}$ (or equivalently that $N(r)$ has infinite order); it suffices to prove the theorem for this twist, since we can untwist each member of the resulting compatible system to deduce the theorem for the original $r$. In any case, since $G_r$ contains $\mr{Spin}_{2n+1}$, $\op{spin}(r)$ is irreducible, as is $\op{spin}(r|_{\Gamma_{K^+}})$ for any finite extension $K^+/F^+$. By the independence-of-$\ell$ observation above, it also follows that $\op{spin}(r_{\tilde{\pi}_{K^+}, \iota_{\ell}})$ is irreducible for any $\iota_{\ell}$ and any $L^+ \supset K^+ \supset F^+$ with $L^+/K^+$ solvable. The argument of \cite[Theorem 5.5.1]{blggt} then immediately implies that the virtual representation $R_{\iota_{\ell}}$ is an actual representation and is moreover irreducible. (Note that \cite[Theorem A]{ks} does not establish
local-global compatibility at all unramified primes: by (ii) of \emph{loc.cit.}, the semisimple conjugacy class of $\iota^{-1}R(\frob{v})^{\mr{ss}}$ agrees with that of $\iota_{\ell}^{-1}R_{\iota_{\ell}}(\frob{v})^{\mr{ss}}$ at those places $v$ such that for all $j$, $\tilde{\pi}_{L_j^+}$ is unramified at $v$ and $v$ is not above $2$ or any rational prime that is ramified in $L_j^+$. This is why we only have compatibility of Frobenii at all but finitely many places $v$ at which $r$ is unramified.)

We now show that each $R_{\iota_{\ell}}$ factors through $\mr{GSpin}_{2n+1}$. Observe that the algebraic monodromy groups $G_r \cong G_R$ are connected, so \cite[Proposition 6.14]{larsen-pink:lind} shows that each $G_{R_{\iota_{\ell}}}$ is connected. Since $R$ is essentially self-dual, so is each $R_{\iota_{\ell}}$, and the unique (by irreducibility) similitude characters $\psi_{\iota_{\ell}}$ of each $R_{\iota_{\ell}}$ are compatible. For each $\iota_{\ell}$, let $T_{\iota_{\ell}}$ be a maximal torus in $G_{R_{\iota_{\ell}}}$; $T_{\iota_{\ell}}$ comes by construction with an embedding $T_{\iota_{\ell}} \to \mr{GL}_{2^n}$, and by \cite[Proposition 6.12]{larsen-pink:lind}, the isomorphism classes of these embeddings of maximal tori are independent of $\iota_{\ell}$. Since the $\psi_{\iota_{\ell}}$ are compatible, the same is true of the embeddings $T_{\iota_{\ell}} \cap G_{R_{\iota_{\ell}}}^{\mr{der}} \to \mr{GL}_{2^n}$. We now apply \cite[Theorem 4]{larsen-pink:invdim} to these embeddings of the maximal tori in the semsimple groups $G_{R_{\iota_{\ell}}}^{\mr{der}}$: 
the only ``basic similarity classes" in the sense of \textit{loc. cit.} relevant to the type $(G_r^{\mr{der}} \subset \mr{GL}_{2^n})= (\mr{Spin}_{2n+1}, \op{spin})$ are products $(\prod_{i=1}^k \mr{Spin}_{2n_i+1}, \boxtimes_i \op{spin}_{2n_i+1})$ for some decomposition $n_1+ \cdots +n_k= n$, where $\op{spin}_{2n_i+1}$ denotes the appropriate spin representation. This external product is simply the restriction of the $2^n$-dimensional spin representation to $\prod \mr{Spin}_{2n_i+1} \subset \mr{Spin}_{2n+1}$, so in any case $G_{R_{\iota_{\ell}}}^{\mr{der}} \subset \mr{GL}_{2^n}$ factors (up to conjugacy) through $\mr{Spin}_{2n+1} \subset \mr{GL}_{2^n}$. It follows that $R_{\iota_{\ell}}$ factors as $\op{spin}(r_{\iota_{\ell}})$ for some $r_{\iota_{\ell}} \colon \Gamma_{F^+} \to \mr{GSpin}_{2n+1}(\bQl)$. Moreover, since $R_{\iota_{\ell}}|_{\Gamma_{L^+}}$ is isomorphic to $\op{spin}(r_{\tilde{\pi}, \iota_{\ell}})$, each member of the compatible system $\{R_{\iota_{\ell}}\}_{\iota_{\ell}}$ of $\Gamma_{F^+}$-representations contains the image under the spin representation of a regular unipotent element of $\mr{Spin}_{2n+1}$. By \cite[Lemma 3.5]{bcempp}, each $r_{\iota_{\ell}}$ then has image containing a regular unipotent element. We deduce from Dynkin's theorem that each algebraic monodromy group $G_{r_{\iota_{\ell}}}$ contains $\mr{Spin}_{2n+1}$. 

That the collection $\{r_{\iota_{\ell}}\}_{\iota_{\ell}}$ forms a compatible system in the sense of $\mr{GSpin}_{2n+1}$-valued representations will now follow from \cite[Lemma 1.3]{ks} and the fact that $N, \op{std}, \op{spin}$ form a fundamental set of representations of $\mr{GSpin}_{2n+1}$. Indeed, by construction the $\op{spin}(r_{\iota_{\ell}})$ are compatible; $N(r_{\iota_{\ell}})$ can be read off (see \cite[Lemma 0.1]{ks}) as the similitude character of the essentially self-dual representation $R_{\iota_{\ell}}$, and these are also clearly compatible. 
To see compatibility of the representations $\op{std}(r_{\iota_{\ell}})$, note that by a standard plethysm 
\[
R_{\iota_{\ell}}^{\otimes 2}\otimes N(r_{\iota_{\ell}})^{-1} \cong \bigoplus_{i=0}^n \wedge^i \op{std}(r_{\iota_{\ell}}).
\]
The representations on the left-hand side of this isomorphism, for varying $\iota_{\ell}$, form a compatible system, so the direct sum on the right-hand side does as well. At the same time, $\op{std}(r)$ is already known to belong a compatible system by \cite[Theorem 5.5.1]{blggt}, and consequently so does each $\wedge^i \op{std}(r)$. Let us denote these compatible systems by $\{\wedge^i \tau_{\iota_{\ell}}\}_{\iota_{\ell}}$. By uniqueness of the compatible system containing a given representation, we deduce an isomorphism
\[
\bigoplus_{i=0}^n \wedge^i \op{std}(r_{\iota_{\ell}}) \cong \bigoplus_{i=0}^n \wedge^i \tau_{\iota_{\ell}}.
\]
Since $\op{std}(r_{\iota_{\ell}})$ and $\tau_{\iota_{\ell}}$ both have algebraic monodromy group $\mr{SO}_{2n+1}$, all of these wedge powers (as $i$ varies) are irreducible of different dimensions, and we conclude that $\op{std}(r_{\iota_{\ell}}) \cong \tau_{\iota_{\ell}}$ for all $\iota_{\ell}$. In other words, the representations $\op{std}(r_{\iota_{\ell}})$ do indeed form a compatible system.

This concludes the proof of the theorem, modulo Theorem \ref{solvabledescent}.
\end{proof}

The following theorem is, for $\mr{GL}_N$, a standard consequence of cyclic prime degree descent and the existence of automorphic Galois representations. 
\begin{thm}\label{solvabledescent}
Let $F^+$ be a totally real field, and let $r \colon \Gamma_{F^+} \to \mr{GSpin}_{2n+1}(\bQl)$ be a continuous representation such that for some finite totally real extension $L^+/F^+$, $r|_{\Gamma_{L^+}}$ is equivalent to $r_{\tilde{\pi}, \iota_{\ell}}$ for some isomorphism $\iota_{\ell} \colon \cmplx \xrightarrow{\sim} \bQl$ and some cuspidal automorphic representation $\tilde{\pi}$ of $\mr{GSp}_{2n}(\mathbb{A}_{L^+})$ satisfying:
\begin{enumerate}
    \item $\tilde{\pi}$ is L-cohomological with respect to the base-change $\xi_{L^+}$ of an irreducible algebraic representation $\xi$ of $\mr{GSp}_{2n}(F^+_{\infty})$.
    \item At some finite place $v_{\mr{St}}$ of $F^+$, 
    $\tilde{\pi}_{v_{\mr{St}}}$ is isomorphic to a twist of the Steinberg representation of $\mr{GSp}_{2n}(L^+_{v_{\mr{St}}})= \prod_{v \vert v_{\mr{St}}} \mr{GSp}_{2n}(L^+_v)$.
\end{enumerate}
Then for any intermediate field $L^+ \supset K^+ \supset F^+$ such that $L^+/K^+$ is a solvable Galois extension, there exists a cuspidal automorphic representation $\tilde{\pi}_{K^+}$ of $\mr{GSp}_{2n}(\mathbb{A}_{K^+})$, with $\tilde{\pi}_{K^+}$ L-cohomological with respect to $\xi_{K^+}$ and isomorphic to a twist of the Steinberg representation above $v_{\mr{St}}$, such that $r_{\tilde{\pi}_{K^+}, \iota_{\ell}} \cong r|_{\Gamma_{K^+}}$.
\end{thm}
\begin{proof}
By induction we reduce to the case of $L^+/K^+$ cyclic of prime degree, showing at each step that there is a cuspidal automorphic representation $\tilde{\pi}_{K^+}$ of $\mr{GSp}_{2n}(\mathbb{A}_{K^+})$ satisfying the analogues over $K^+$ of conditions (1) and (2) (i.e., the hypotheses of \cite[Theorem A]{ks}), and such that the associated $\ell$-adic representation $r_{\tilde{\pi}_{K^+}, \iota_{\ell}}$ is equivalent to $r|_{\Gamma_{K^+}}$. 

Let $\pi$ be any cuspidal automorphic constituent of $\tilde{\pi}|_{\mr{Sp}_{2n}(\mathbb{A}_{L^+})}$, and let $\Pi$ be the functorial transfer of $\pi$ to a cuspidal automorphic representation of $\mr{GL}_{2n+1}(\mathbb{A}_{L^+})$. The cuspidality of $\Pi$ follows from \cite[Corollary 2.2]{ks}. Moreover, $\Pi_w$ is isomorphic to the Steinberg representation for all $w \vert v_{\mr{St}}$ (\cite[Proposition 8.2]{ms}), and it is regular L-algebraic with archimedean L-parameters $\mr{std} \circ \mr{rec}_w(\tilde{\pi})$ for all $w \vert \infty$ (see the discussion surrounding \cite[Theorem 2.4]{ks}). Likewise, the Satake parameters of $\Pi$ are those of $\tilde{\pi}$ composed with $\mr{std}$, so by the strong multiplicity one theorem (\cite{jacquet-shalika}) $\Pi^{\sigma}=\Pi$. Cyclic descent of prime degree (\cite{arthur-clozel}) implies that $\Pi$ descends to a cuspidal automorphic representation $\Pi_{K^+}$ of $\mr{GL}_{2n+1}(\mathbb{A}_{K^+})$. It follows that $\Pi_{K^+}$ is regular L-algebraic with the same archimedean L-parameters (under the identification $W_{K^+_v}= W_{L^+_w}$ for all $w \vert v \vert \infty$) as $\Pi$, and that $\Pi_{K^+, v_{\mr{St}}}$ is isomorphic to a twist of the Steinberg representation. We claim that $\Pi_{K^+}$ may be chosen to be self-dual. The description of the fibers of cyclic base-change in \cite{arthur-clozel} implies (by self-duality of $\Pi$) that $\Pi_{K^+}$ is essentially self-dual: $\Pi_{K^+} \cong \Pi_{K^+}^{\vee} \otimes \psi$ for some Hecke character $\psi$ of $K^+$. By \cite[Theorem 2.1]{stp:sign}, $\psi_v(-1)$ is independent of $v \vert \infty$, and so by \cite[Theorem 1.2]{sh11}, as formulated in \cite[Theorem 2.1.1]{blggt}, $\Pi_{K^+}$ has an associated Galois representation $r_{\Pi_{K^+}, \iota_{\ell}} \colon \Gamma_{K^+} \to \mr{GL}_{2n+1}(\bQl)$. Both $r_{\Pi_{K^+}, \iota_{\ell}}$ and $\mr{std}(r|_{\Gamma_{K^+}})$ descend $\mr{std}(r_{\tilde{\pi}, \iota_{\ell}})$, so (by irreducibility) there is a character $\eta \colon \mr{Gal}(L^+/K^+) \to \bQl^\times$ such that $r_{\Pi_{K^+}, \iota_{\ell}} \cong \mr{std}(r|_{\Gamma_{K^+}}) \cdot \eta$. Thus
\[
r_{\Pi_{K^+}, \iota_{\ell}}^\vee \cong \mr{std}(r|_{\Gamma_{K^+}})^\vee \cdot \eta^{-1} \cong \mr{std}(r|_{\Gamma_{K^+}}) \cdot \eta^{-1} \cong r_{\Pi_{K^+}, \iota_{\ell}} \otimes \eta^{-2}.
\]
Replacing $\Pi_{K^+}$ by its twist $\Pi_{K^+} \otimes \eta^{-1}$ (note that $\eta$ is finite-order, so we may view it as a Hecke character), we may and do assume that $\Pi_{K^+}$ is self-dual.

Now, by \cite{art13}, there is a cuspidal automorphic representation $\pi_{K^+}$ of $\mr{Sp}_{2n}(\mathbb{A}_{K^+})$ transferring to $\Pi_{K^+}$, and more precisely satisfying the conclusions of Proposition \ref{descent}. We can then apply Proposition \ref{lift aut} to lift $\pi_{K^+}$ to a cuspidal automorphic representation of $\tilde{\pi}_{K^+}$ satisfying the conclusions of \textit{loc. cit.}, and in particular having an associated Galois representation $r_{\tilde{\pi}_{K^+}, \iota_{\ell}} \colon \Gamma_{K^+} \to \mr{GSpin}_{2n+1}(\bQl)$. Finally, we compare $r|_{\Gamma_{K^+}}$ with $r_{\tilde{\pi}_{K^+}, \iota_{\ell}}$. Their compositions with $P \colon \mr{GSpin}_{2n+1} \to \mr{SO}_{2n+1}$ are $\mr{SO}_{2n+1}$-conjugate upon restriction to $\Gamma_{L^+}$ (taking note of the uniqueness statement in \cite[Proposition B.1]{ks}), so we may assume they have equal restriction to $\Gamma_{L^+}$. We claim that $P(r|_{\Gamma_{K^+}})$ and $P(r_{\tilde{\pi}_{K^+}, \iota_{\ell}})$ are necessarily equal. Indeed, any two homomorphisms $r_1, r_2 \colon \Gamma_{K^+} \to \mr{SO}_{2n+1}(\bQl)$ that are $\mr{GL}_{2n+1}$-irreducible and become equal after restriction to $\Gamma_{L^+}$ must be equal: $r_1=r_2 \cdot \eta$ for some character $\eta$ of $\mr{Gal}(L^+/K^+)$, by the irreducibility, and then $\eta$ necessarily satisfies $\eta^{2n+1}=1$ and $\eta^2=1$, hence $\eta=1$.

Thus $r|_{\Gamma_{K^+}} \otimes \chi = r_{\tilde{\pi}_{K^+}, \iota_{\ell}}$ for some character $\chi \colon \Gamma_{K^+} \to \bQl^\times$, which is finite-order since its Hodge-Tate weights are all zero. Regarding $\chi$ as a Hecke character, we can then replace $\tilde{\pi}_{K^+}$ with $\tilde{\pi}_{K^+} \otimes \chi^{-1}$, and (having made this replacement) we as desired have arranged that $r_{\tilde{\pi}_{K^+, \iota_{\ell}}}= r|_{\Gamma_{K^+}}$, and that $\tilde{\pi}_{K^+}$ still satisfies the hypotheses of \cite[Theorem A]{ks}. 
\end{proof}
This concludes the proof of Theorem \ref{compsystem}, but the following complementary result on automorphic descent may be of some interest:
\begin{cor}\label{descentinv}
Let $L^+/K^+$ be a solvable extension of totally real fields. Let $\tilde{\pi}$ be a cuspidal automorphic representation of $\mr{GSp}_{2n}(\mathbb{A}_{L^+})$ satisfying the hypotheses of \cite[Theorem A]{ks}, i.e. L-cohomological of some weight $\xi$ and a twist of the Steinberg representation at some finite place $w_{\mr{St}}$. Assume that $\tilde{\pi}^\sigma \cong \tilde{\pi}$ for all $\sigma \in \mr{Gal}(L^+/K^+)$. Then there is a $\mr{Gal}(L^+/K^+)$-invariant finite-order character $\chi \colon \Gamma_{L^+} \to \bQl^\times$ and a cuspidal automorphic representation $\tilde{\pi}_{K^+}$ of $\mr{GSp}_{2n}(\mathbb{A}_{K^+})$ that is unramified at primes below those where $\tilde{\pi}$, $\chi$, and $L^+/K^+$ are unramified, isomorphic to a twist of the Steinberg representation at the place below $w_{\mr{St}}$, and L-cohomological with weight descending that of $\tilde{\pi}_{\infty}$, such that $\tilde{\pi}_{K^+}$ descends $\tilde{\pi} \otimes \chi$ in the following sense: for the three kinds of places just specified the local L-parameters are defined, and in each case the local L-parameter of $\tilde{\pi}_{K^+}$ restricts to that of $\tilde{\pi} \otimes \chi$.

In particular, if $L^+/K^+$ is cyclic, we may descend $\chi$ and therefore descend $\tilde{\pi}$ itself.
\end{cor}
\begin{rmk}
For cyclic extensions, all invariant Hecke characters descend, but the corresponding property fails for solvable extensions. A template for the present corollary is a theorem of Rajan (\cite[Theorem 1]{rajan:solvable}) establishing an analogous result for general cuspidal automorphic representations of $\mr{GL}_N$ over any number field. In particular, Rajan's results show that not every $\tilde{\pi}$ as in the corollary can descend (even for $n=1$), since if $\tilde{\pi}$ descends, its twist by an invariant (but non-descending) character will not descend (but will still satisfy the hypotheses of the corollary): see the uniqueness claim in \cite[Theorem 1]{rajan:solvable}.
\end{rmk}
\begin{proof}
For all $\sigma \in \mr{Gal}(L^+/K^+)$, $\tilde{\pi}^\sigma$ has an associated Galois representation $r_{\tilde{\pi}^\sigma, \iota_{\ell}} \colon \Gamma_{L^+} \to \mr{GSpin}_{2n+1}(\bQl)$, which is equivalent to $r_{\tilde{\pi}, \iota_{\ell}}$ and to $r_{\tilde{\pi}, \iota_{\ell}}^{\sigma}$ by the uniqueness assertion of \cite[Theorem A]{ks} (see also Proposition 5.4 of \textit{loc. cit.}). For each $\sigma \in \Gamma_{K^+}$, let $A_{\sigma}$ be an element of $\mr{GSpin}_{2n+1}(\bQl)$ such that $r_{\tilde{\pi}, \iota_{\ell}}^\sigma(h)= A_{\sigma} r_{\tilde{\pi}, \iota_{\ell}}(h) A_{\sigma}^{-1}$ for all $h \in \Gamma_{L^+}$. We may and do assume that for all $\sigma \in \Gamma_{L^+}$, $A_{\sigma}= r_{\tilde{\pi}, \iota_{\ell}}(\sigma)$, and that $A_{\sigma}$ is defined in general by fixing its values on a fixed set of representatives $\sigma$ in $\Gamma_{K^+}$ of $\mr{Gal}(L^+/K^+)$ and then extending to $\Gamma_{K^+}$ by $A_{\sigma h}= A_{\sigma} A_h= A_{\sigma} r_{\tilde{\pi}, \iota_{\ell}}(h)$ for $h \in \Gamma_{L^+}$. Since the centralizer of $r_{\tilde{\pi}, \iota_{\ell}}$ is the center $Z \subset \mr{GSpin}_{2n+1}$, the assignment $c \colon (\sigma, \tau) \mapsto A_{\sigma \tau}^{-1} A_{\sigma} A_{\tau}$ takes values in $Z(\bQl) \simeq \bQl^\times$, and therefore $\sigma \mapsto A_{\sigma}$ is a continuous homomorphism $\Gamma_{K^+} \to \mr{SO}_{2n+1}(\bQl)$ whose restriction to $\Gamma_{L^+}$ is $P(r_{\tilde{\pi}, \iota_{\ell}})$.

Moreover, $c$ formally satisfies the 2-cocycle relation, and by our normalization of the $A_{\sigma}$, $c(\cdot, \cdot)$ is constant on $\Gamma_{L^+}$-cosets, and in particular is continuous with finite image. It can therefore be regarded (via an identification $\rats/\ints \xrightarrow{\sim} \mu_{\infty}(\bQl)$) as a cohomology class $[c] \in H^2(\Gamma_{K^+}, \rats/\ints)$ for the discrete Galois module $\rats/\ints$. By a theorem of Tate (\cite[Theorem 4]{serre:tate}), this cohomology group vanishes. Writing $c$ as a coboundary then allows us to adjust the choice of $A_{\sigma}$ such that $\sigma \mapsto A_{\sigma}$ defines a continuous homomorphism $r_{K^+} \colon \Gamma_{K^+} \to \mr{GSpin}_{2n+1}(\bQl)$ such that for some finite Galois extension $\widetilde{L}^+$ of $L^+$, $r_{K^+}|_{\Gamma_{\widetilde{L}^+}}$ equals $r_{\tilde{\pi}, \iota_{\ell}}|_{\Gamma_{\widetilde{L}^+}}$. (The reason we do not obtain this equality already for $\Gamma_{L^+}$ is that we do not know that the function realizing $c$ as a coboundary factors through $\mr{Gal}(L^+/K^+)$, and indeed it may not.) Since (as noted in the previous paragraph) $P(r_{K^+}|_{\Gamma_{L^+}})= P(r_{\tilde{\pi}, \iota_{\ell}})$, we at least have $r_{K^+}|_{\Gamma_{L^+}} = r_{\tilde{\pi}, \iota_{\ell}} \otimes \chi$ for some character $\chi$, necessarily factoring through $\mr{Gal}(\widetilde{L}^+/L^+)$. Conjugating by elements $\sigma \in \mr{Gal}(L^+/K^+)$, we see that $r_{\tilde{\pi}, \iota_{\ell}} \otimes \chi^{\sigma -1}$ is equivalent to $r_{\tilde{\pi}, \iota_{\ell}}$ for all $\sigma \in \Gamma_{K^+}$. Since the image contains a regular unipotent element, we deduce from \cite[Proposition 5.2]{ks} that $\chi^{\sigma}= \chi$ for all $\sigma \in \mr{Gal}(L^+/K^+)$.

Thus, $\tilde{\pi} \otimes \chi$ and $r_{K^+}$ satisfy the hypothese of Theorem \ref{solvabledescent}, and the automorphic representation $\tilde{\pi}_{K^+}$ produced by \textit{loc. cit.} satisfies the conclusions of the present corollary.
\end{proof}

\end{document}